\documentclass[a4paper]{amsart}
\usepackage{amsmath,caption,booktabs}
\numberwithin{equation}{section}
\usepackage{amsthm}
\usepackage{amsfonts}
\usepackage{amssymb}
\usepackage{mathrsfs}
\usepackage{tikz}
\usepackage{pdfpages}
\usepackage{environ}
\usepackage{elocalloc}
\usepackage{url}
\usepackage{caption}
\usepackage{graphicx}
\usepackage{CJKutf8}
\usepackage[normalem]{ulem}
\usepackage{xcolor}
\usepackage{etoolbox}
\usepackage{caption}
\usepackage{mathtools}
\usepackage{dcpic}
\usepackage{tikz-cd}
\usepackage[bottom]{footmisc}
\usepackage{hyperref}
\usepackage{enumitem}
\usepackage{stmaryrd}
\usetikzlibrary{positioning}
\usetikzlibrary{arrows,chains,positioning,scopes,quotes}
\usetikzlibrary{decorations.markings}

\title[Every finite graph arises as the singular set of calibrated currents]{Every finite graph arises as the singular set of a compact $3$-d calibrated area minimizing surface}
\author{Zhenhua Liu}
\date{}
\dedicatory{Dedicated to Xunjing Wei}
\begin{document}

	\setlength{\parskip}{0.07 cm}%
	\setlength{\parindent}{0pt}%
	\newcommand{\ai}{\alpha}
	\newcommand{\be}{\beta}
	\newcommand{\Ga}{\Gamma}
	\newcommand{\ga}{\gamma}	
	\newcommand{\de}{\delta}
	\newcommand{\De}{\Delta}
	\newcommand{\e}{\epsilon}
	\newcommand{\lam}{\lambda}
	\newcommand{\Lam}{\Lamda}
	\newcommand{\om}{\omega}
	\newcommand{\Om}{\Omega}
	\newcommand{\si}{\sigma}
	\newcommand{\Si}{\Sigma}
	\newcommand{\vp}{\varphi}
	\newcommand{\rh}{\rho}
	\newcommand{\ta}{\theta}
	\newcommand{\Ta}{\Theta}
	\newcommand{\W}{\mathcal{O}}
	\newcommand{\ps}{\psi}
	
	\newcommand{\mf}[1]{\mathfrak{#1}}
	\newcommand{\ms}[1]{\mathscr{#1}}
	\newcommand{\mb}[1]{\mathbb{#1}}
	\newcommand{\cd}{\cdots}
	
	\newcommand{\s}{\subset}
	\newcommand{\es}{\varnothing}
	\newcommand{\cp}{^\complement}
	\newcommand{\bu}{\bigcup}
	\newcommand{\ba}{\bigcap}
	\newcommand{\co}{^\circ}
	\newcommand{\ito}{\uparrow}
	\newcommand{\dto}{\downarrow}

	\newcommand{\ti}[1]{\tilde{#1}}
	\newcommand{\gr}{\textnormal{gr}}
	\newcommand{\la}{\langle}
	\newcommand{\ra}{\rangle}
	\newcommand{\ov}[1]{\overline{#1}}
	\newcommand{\no}[1]{\left\lVert#1\right\rVert}
	\newcommand{\cu}[1]{\left\llbracket#1\right\rrbracket}
	\DeclarePairedDelimiter{\cl}{\lceil}{\rceil}
	\DeclarePairedDelimiter{\fl}{\lfloor}{\rfloor}
	\DeclarePairedDelimiter{\ri}{\la}{\ra}

	\newcommand{\du}{^\ast}
	\newcommand{\pf}{_\ast}
	\newcommand{\is}{\cong}
	\newcommand{\n}{\lhd}
	\newcommand{\m}{^{-1}}
	\newcommand{\ts}{\otimes}
	\newcommand{\ip}{\cdot}
	\newcommand{\op}{\oplus}
	\newcommand{\xr}{\xrightarrow}
	\newcommand{\xla}{\xleftarrow}
	\newcommand{\xhl}{\xhookleftarrow}
	\newcommand{\xhr}{\xhookrightarrow}
	\newcommand{\mi}{\mathfrak{m}}
	\newcommand{\wi}{\widehat}
	\newcommand{\sch}{\mathcal{S}}

	\newcommand{\w}{\wedge}
	\newcommand{\X}{\mathfrak{X}}
	\newcommand{\pd}{\partial}
	\newcommand{\dx}{\dot{x}}
	\newcommand{\dr}{\dot{r}}
	\newcommand{\dy}{\dot{y}}
	\newcommand{\dth}{\dot{theta}}
	\newcommand{\pa}[2]{\frac{\pd #1}{\pd #2}}
	\newcommand{\na}{\nabla}
	\newcommand{\dt}[1]{\frac{d#1}{d t}\bigg|_{ t=0}}
	\newcommand{\ld}{\mathcal{L}}

	\newcommand{\N}{\mathbb{N}}
	\newcommand{\R}{\mathbb{R}}
	\newcommand{\Z}{\mathbb{Z}}
	\newcommand{\Q}{\mathbb{Q}}
	\newcommand{\C}{\mathbb{C}}
	\newcommand{\bh}{\mathbb{H}}
	
	\newcommand{\lix}{\lim_{x\to\infty}}
	\newcommand{\li}{\lim_{n\to\infty}}
	\newcommand{\infti}{\sum_{i=1}^{\infty}}
	\newcommand{\inftj}{\sum_{j=1}^{\infty}}
	\newcommand{\inftn}{\sum_{n=1}^{\infty}}	
	\newcommand{\snz}{\sum_{n=-\infty}^{\infty}}	
	\newcommand{\ie}{\int_E}
	\newcommand{\ir}{\int_R}
	\newcommand{\ii}{\int_0^1}
	\newcommand{\sni}{\sum_{n=0}^\infty}
	\newcommand{\ig}{\int_{\ga}}
	\newcommand{\pj}{\mb{P}}
	\newcommand{\dia}{\textnormal{diag}}
	
	\newcommand{\io}{\textnormal{ i.o.}}
	\newcommand{\aut}{\textnormal{Aut}}
	\newcommand{\out}{\textnormal{Out}}
	\newcommand{\inn}{\textnormal{Inn}}
	\newcommand{\mult}{\textnormal{mult}}
	\newcommand{\ord}{\textnormal{ord}}
	\newcommand{\F}{\mathcal{F}}
	\newcommand{\V}{\mathbf{V}}	
	\newcommand{\II}{\mathbf{I}}
	\newcommand{\ric}{\textnormal{Ric}}
	\newcommand{\sef}{\textnormal{II}}
	
	\newcommand{\wh}{\Rightarrow}
	\newcommand{\eq}{\Leftrightarrow}
	
	\newcommand{\eqz}{\setcounter{equation}{0}}
	\newcommand{\se}{\subsection*}
	\newcommand{\ho}{\textnormal{Hom}}
	\newcommand{\ds}{\displaystyle}
	\newcommand{\tr}{\textnormal{tr}}
	\newcommand{\id}{\textnormal{id}}
	\newcommand{\im}{\textnormal{im}}
	\newcommand{\ev}{\textnormal{ev}}
	
	\newcommand{\gl}{\mf{gl}}
	\newcommand{\sll}{\mf{sl}}
	\newcommand{\su}{\mf{su}}
	\newcommand{\so}{\mf{so}}
	\newcommand{\ad}{\textnormal{ad}}
	\newcommand\res{\mathop{\hbox{\vrule height 7pt width .3pt depth 0pt
				\vrule height .3pt width 5pt depth 0pt}}\nolimits}
	
	\newcommand\ser{\mathop{\hbox{\vrule height .3pt width 5pt depth 0pt
				\vrule height 7pt width .3pt depth 0pt}}\nolimits}
	
	\theoremstyle{plain}
	\newtheorem{thm}{Theorem}[section]
	\newtheorem{lem}[thm]{Lemma}
	\newtheorem{prop}[thm]{Proposition}
	\newtheorem*{cor}{Corollary}
	\newtheorem*{pro}{Proposition}
	
	\theoremstyle{definition}
	\newtheorem{defn}{Definition}[section]
	\newtheorem{conj}{Conjecture}[section]
	\newtheorem{exmp}{Example}[section]
	
	\theoremstyle{remark}
	\newtheorem{rem}{Remark}
	\newtheorem*{note}{Note}
	
	\raggedbottom
	
	\begin{abstract}
		Given any (not necessarily connected) combinatorial finite graph and any compact smooth $6$-manifold  $M^6$ with the third Betti number $b_3\not=0$, we construct a calibrated 3-dimensional homologically area minimizing surface on $M$ equipped in a smooth metric $g$, so that the singular set of the surface is precisely an embedding of this finite graph. Moreover, the calibration form near the singular set is a smoothly $GL(6,\R)$ twisted special Lagrangian form. The constructions are based on some unpublished ideas of Professor Camillo De Lellis (\cite{DL}) and Professor Robert Bryant (\cite{RB}).
	\end{abstract}
	\maketitle
	\tableofcontents
	\section{Introduction}
	There are many different formulations of area-minimizing surfaces. In this paper, area-minimizing surfaces refer to area-minimizing integral currents, which roughly speaking are oriented surfaces counted with multiplicity, minimizing the area functional. Calling them surfaces is justified thanks to Almgren's Big Theorem (\cite{A}) and De Lellis-Spadaro's recent proof (\cite{DS1}\cite{DS2}\cite{DS3}). Their results show that $n$-dimensional area minimizing integral currents are smooth manifolds outside of a singular set of dimension at most $n-2.$ (In the codimension 1 case, the dimension of the singular set can be reduced to $n-7$ by \cite{HF1}.)
	
	Area-minimizing surfaces arise naturally in special holonomic geometries, thanks to the fundamental theorem of calibrated geometries by Harvey and Lawson (\cite{HLg}), which says that calibrated surfaces are area-minimizing. Calibrated surfaces like special Lagrangians and associative submanifolds have received a lot of attention in recent years. However, we do not know much about the nature of the singular set of these surfaces, which arise even as limits of sequences of smooth calibrated surfaces. 
	
	With the marvelous regularity theorems by Almgren-De Lellis-Spadaro in mind, a natural next step is to ask what more we can say about the singular set. We know that $2$-dimensional area minimizing integral currents are classical branched minimal immersed surfaces by \cite{SC}\cite{DSS1}\cite{DSS2}\cite{DSS3}, and all the tangent cones are unique by \cite{DSS0}\cite{BW}. 
	
	When we come to $3$-dimensional area minimizing integral currents, there is no such clear-cut description of the singular set. We know the singular set has Hausdorff dimension at most $1.$ We know the singular set in the $1$-strata of Almgren stratification is $1$-rectifiable (\cite{NV}), but not much about the top dimensional strata. 
	
	It is natural to ask what set can be the singular set and what is the structure of the surface around the singular set. Professor Robert Bryant has constructed special Lagrangians with the singular set being complete real analytic curves in a $2$-plane (\cite{RB1}). In \cite{ZL}, we construct $3$-dimensional calibrated area-minimizing surfaces with the $1$-strata of the singular set being any closed set of a circle. 
	
	However, both Bryant's and the author's examples have only intersection-type singularities. In this manuscript, using a smoothly $GL(6,\R)$ twisted version of the special Lagrangian calibration, we prove the following result. The singular set has genuine conical points at the vertices of the graphs.
	\begin{thm}\label{anyg}
		Let $M^6$ be a compact closed smooth (not necessarily orientable) manifold with the third Betti number $b_3\not=0.$ For any (not necessarily connected) finite graph $G$ in the combinatorial sense, there exists a 3-dimensional calibrated homologically area-minimizing surface $N$, so that the singular set of $N$ is an embedding of $G$. 
	\end{thm}
	\begin{rem}The surface $N$ is irreducible in the sense of Definition \ref{irr} and  connected. 
	\end{rem}
	\begin{rem}
		The above theorem holds for closed $M^d$ with any $d\ge 6$ and $H_3(M,\R)\not=0$ dropping the calibration but retaining area-minimization. 
	\end{rem}
	\begin{rem}\label{twst}
		The precise meaning of using smoothly $GL(6,\R)$ twisted version of the special Lagrangian calibration is as follows. Near the singular set of the surfaces, both the metric and the calibration form are smoothly varying $GL(6,\R)$ pull-backs of the standard metric on $\C^3$ and the standard special Lagrangian form on $\C^3.$
	\end{rem}
	\begin{rem}
		This work and the main theorems in \cite{ZL} use totally different methods and are complementary to each other. For $3$-dimensional surfaces, the work in \cite{ZL} alone cannot prescribe a finite graph as the singular set. We will sharpen the results in \cite{ZL} using the above theorem in future works.
	\end{rem}
	\begin{rem}
		Recently, Professor Leon Simon posted the groundbreaking construction of stable minimal hypersurfaces with Cantor set singularities \cite{LS}, which has cited a preprint version of this manuscript. Indeed, it is Professor Simon's work that encourages us to prove the now presented version of Theorem \ref{anyg}. We comment on the differences between the two. The first difference is that the minimal surfaces in his work are codimension 1, while ours are of higher codimension, thus permitting relatively larger singular sets. His approach is PDE based, and we use soft calibration arguments. It is unknown if the stable minimal hypersurfaces in his construction are minimizing, while our construction is calibrated, thus minimizing. For his construction, when the largest dimension of the singular set is one, i.e., on minimal hypersurfaces in dimension $\R^8$, the singular set is always a subset of a line, while our construction can give any finite graph, which can also be quite complicated and sometimes cannot even be embedded into any $\R^2.$ 
	\end{rem}
	\subsection{Sketch of the proof}
	The idea is as follows. We first illustrate the simple case of getting a line segment of singularities, then proceed to the more general case of finite graphs.
	
	For a special Lagrangian cone $C$, we can take the union with some special Lagrangian plane $P$. Such unions are also area-minimizing by calibration, and with a careful choice, the singular set can be a ray $\ga^-$. Professor Robert Bryant first noticed this behavior property for $G_2$ cones (\cite{RB}).
	 
Professor Camillo De Lellis has pointed out that one can glue the reflected copies of $C\cup P$ to the original union using an implicit function theorem argument (\cite{DL}). Roughly speaking, one can first apply a special unitary rotation $\rh$, so that the singular set of $\rh(C)\cup \rh(P)$ is a ray $\ga^+=\rh(C)\cap\rh(P)$ in the opposite direction to the ray $\ga=C\cap P$. Now translate $\rh(C)\cup\rh(P)$ along the direction of a vector $v$ parallel to $\ga.$ Then $\ga^-+v$ intersect $\ga^+$ along a line segment $\ga$. Now we want to glue $C$ to $\rh(C)+v$ and $P$ to $\rh(P)+v$ near the midpoint of $\ga$ while preserving area-minimizing properties.

By construction, the glued surfaces decompose into two very flat pieces near the midpoint of $\ga$. By an implicit function argument, we can prove that $GL(6)$ pushforwards of the special Lagrangian form and pushforwards of the standard metric can be used to calibrate the glued-together surfaces. This is one of the main technical difficulties as we need to deal with the two pieces together simultaneously while making the form closed.

At this stage, the surface is reducible to two pieces and has two boundary components. We first fill in each boundary with smooth $3$-manifolds and use a neck to connect sum the two components. Then use transversality to ensure only isolated transverse intersections are added to the singular set. Now replace each intersection with a neck. Up to this step, everything can be embedded into a $6$-sphere. Finally, take a simultaneous connected sum with nontrivial homology classes on the target manifold. By the results in \cite{YZa}, in each step, we can modify or extend the forms and the metrics to retain a calibration and extend globally in the final step.
		
	It is trickier to get an arbitrary graph. Instead of just using one cone and one plane and their reflections as a model, we can use one cone and a family of planes that pairwise intersect only at the origin and each of which intersects the cone along a ray. This can give an arbitrary number of intersecting rays, thus realizing an arbitrary degree of a vertex in a graph. 
	
	We then place the vertices along a line, called the book spine, in $\R^6,$ and construct roughly a book embedding of the graph. The idea is that by hand constructing diffeomorphisms, we can rotate the realizing collection of the vertices we want into the compatible configuration with an edge to be filled in, and push all the things in between into slices of $\R^5,$ called pages, that contain this line. We can then use the gluing argument for getting a line segment to fill in the edges. Finally, we can reverse the diffeomorphism, and pull back to the preimage of this diffeomorphism. Thus, the connecting edges stay in a twisted way on the different pages.
	
	The sections will be arranged as follows. First, we will show how to construct one line segment of singularities. Then we show how to modify the argument to produce a graph of singularities. Finally, we will show how to make the examples complete and homologically minimizing. Then we will give a technical enhancement of the construction above. In the appendix, we will give proof of several lemmas that are calculation based.
	\section*{Acknowledgements}
	The author acknowledges the support
	of the NSF through the grant FRG-1854147.
	I cannot thank my advisor Professor Camillo De Lellis enough for his unwavering support while I have been recovering from illness. I would also like to thank him for giving me this problem and for always generously sharing his ideas with me. I feel so lucky that I have Camillo as my advisor. I would also like to thank Professor Robert Bryant for sharing his insights on calibrations and associative cones. Without their help, this project would not have existed. I would also like to thank Dr. Donghao Wang for answering my questions on differential topology. Last but not least, I want to thank Professor Costante Bellettini, Professor Mark Haskins, Professor Fanghua Lin, Professor Siqi He, Professor Leon Simon, Professor Song Sun, and Professor Yongsheng Zhang (in alphabetical order of last name) for their interest in this work.
	\section{Notations and basic definitions}
	In this section, we will collect several notation conventions. 
	\subsection{General set up}
	Our base space is $\C^3.$ Equip $\R^6\equiv \C^3$ with the standard complex structure $J.$ The coordinates are denoted $x_1,x_2,x_3,y_1,y_2,y_3$ for $\R^6$, with $z_j=x_j+iy_j$ for the $\C^3$. Moreover, $\pd_{y_j}=J\pd_{x_j}.$ Let the standard special Lagrangian form be $\phi=\Re dz_1\w dz_2\w dz_3.$ The main reference we use for special Lagrangian geometries is the classical paper \cite{HLg}.

We will frequently speak of rays, by which we mean a smooth affine image of $[0,\infty)$ onto a straight line in Euclidean space. The direction of a ray is defined to be vectors parallel to this ray.
	
	When we mention a surface $T$, we mean an integral current $\cu{T}$. For a comprehensive introduction to integral currents, the standard references are \cite{LS1} and \cite{HF}. We will adhere to their notations. Our manuscript mostly focuses on the differential geometric side and in fact, no a priori knowledge of currents is needed. Every time we mention a current, the reader can just assume it to be a sum of chains representing oriented surfaces with singularities. We will use the following definition of the irreducibility of currents.
	\begin{defn}\label{irr}
		A closed integral current $T$ is irreducible in $U,$ if we cannot write $T=S+W,$ with $\pd S=\pd W=0,$ $S,W$ nonzero and $S$ not an integer multiple of $W.$
	\end{defn}
	Also, we will use the differential geometry convention of completeness. An integral current $T$ is complete if $\pd T=0$ and $T$ has compact support. 
	
	When we refer to $x_j$-axis, (or $y_j$-axis,) it means the $1$ dimensional real vector space spanned by $\pd_{x_j}$ (or $\pd_{y_j},$ respectively.) We will use $\cu{x_1x_2x_3},\cu{y_1y_2x_3},$ etc., to denote the special Lagrangian current associated with the $x_1x_2x_3$-plane, $y_1y_2x_3$-plane, etc. 
	
	For any two points in $p,q$ in a Euclidean space, $[p,q]$ denotes the line segment connecting them.	
	
	For any smooth manifold $N$ with boundary, we will use $B_r(N)$ to denote a tubular neighborhood of radius $r$ around $N.$
	\subsection{Calibrations}
We assume knowledge of the notions of comass and calibrations (Section II.3 and II.4 in \cite{HLg}). The primary reference is \cite{HLg}. The most important thing to keep in mind is the fundamental theorem of calibrations (Theorem 4.2 in \cite{HLg}), i.e., calibrated currents are area-minimizing with homologous competitors.  We will use it many times without explicitly citing it. 
	\subsection{Planes and matrices}
	We will use $\ov{Gr}(3,6)$ to denote the Grassmannian of oriented $3$-planes in $\R^6.$ The notation $\textnormal{diag}\{a_1,a_2,a_3\}$ denotes the diagonal matrix in $GL(3,\C)$ with diagonal elements consisting of $a_j$ in the $j$-th diagonal term.
	
	Let $P\in GL(3,\C)$ be a matrix. We will also use $P$ to denote the $\R$-span of its rows. We give $P$ the orientation induced by $v_1\w v_2\w v_3,$ with $v_j$ being the $j$-th row. In this way, $P$ also denotes an element in $\ov{Gr}(3,6).$
	
	Note that an element $P\in U(3)$ (or $SU(3)$) denotes an oriented Lagrangian plane (or special Lagrangian, respectively). To see this, by \cite{HLg}, a plane $P$ is Lagrangian (special Lagrangian) if it is the image of the $x_1x_2x_3$-plane under an element in $U(3)$ ($SU(3)$, respectively.) The $x_1x_2x_3$-plane is simply $\id= \textnormal{diag}\{1,1,1\}$ in our notation. The claim follows from $P=\id.P.$
	\subsection{Graph theory}
	For a finite graph we will use $v_j$ to denote its vertices and $E_{k,l}=E_{l,k}$ to denote an edge connecting $v_k$ and $v_l.$ The degree of a vertex $\deg v_j$ is defined the standard way, i.e., the number of edges starting from the vertex $v_j$.	 
\begin{defn}
A map $h$ from a finite graph $G$ to $\R^n$ is called an embedding if ti satisfies the following conditions.
\begin{itemize}
\item $h(v_j)\not= h(v_l)$ for $j\not=l$,
\item $h(E_{j,k})$ is a smooth curve connecting $h(v_j),h(v_k)$for all $j,k$,
\item $h(E_{j,k})\cap h(E_{l,m})$ consists of the image of the common vertex of $E_{j,k},E_{l,m}$, and is empty if $E_{j,k},E_{l,m}$ has no common vertex,
\item and $h(E_{j,k}),h(E_{j,l})$ is not tangential to each other at $h(v_j)$ for $k\not=l$.
\end{itemize}
\end{defn}	
We will frequently abuse the notations, and use $v_l,$ $E_{l,m}$, $G$, etc., to denote both the intrinsic notion of vertices, edges, graphs, and their images into $\R^n.$ $B_r(G)$ denotes a tubular neighborhood of radius $r$ around a graph $G.$
	
	Sometimes we will talk about realizing the pair/collection of a vertex $v_j$ of $G.$ By that we mean a certain collection of planes and a cone whose number of intersecting rays is precisely $\deg v_j$. See Section \ref{exa} and \ref{mexa}.
	\subsection{Lemmas about changing coordinates}
	In this subsection, we collect several useful lemmas about changing coordinates. The readers can skip this section for the read.
	
	Here we first state a general change of coordinates lemma. 	Suppose we have a $3$-dimensional special Lagrangian cone $C\s \R^6$ with an isolated singularity and a $3$-dimensional special Lagrangian plane $P,$ so that $C$ and $P$ intersect only at $k$ rays $\ga_0,\cd,\ga_{k-1}$. Moreover, for each intersecting ray, $P$ is not tangential to $C.$ We will show that $k=1,2,4$ cases exist by Lemma \ref{rays}. Thus, we can write the tangent cone $C_q$ to the union $\llbracket P\rrbracket +\llbracket C\rrbracket $ at any point $q$ on $\ga_j\backslash\{0\}$ as  $\llbracket C_q\rrbracket =\llbracket P\rrbracket +\llbracket T_qC\rrbracket $, with $T_q C$ being the oriented tangent plane to $C$ at $q.$ Moreover, the intersection $P\cap T_q C$ is a line. (This directly follows from the fact that we are taking tangent planes of a cone on rays away from zero). Since $C$ is a cone, one can verify that all such $T_q C$ along the same ray coincide with each other.
	\begin{lem}\label{coc}
		For each $\ga_j,$ there exists an element $h\in SU(3)$ so that after changing coordinates using $h$, we can assume that $T_q C$ is given by the $x_1,x_2,x_3$-plane. $P$ is given by the product of the $x_3$-axis with a 2-d special Lagrangian plane $\pi\s \C^2$, the span of the $x_1,y_1,x_2,y_2$ axes, and $\ga_j$ lies on the nonnegative $x_3$-axis. Alternatively, we can choose $\ga_j$ to lie on the nonpositive $x_3$-axis.
	\end{lem}
	\begin{rem}
		The change of coordinate $h$ might be different for each ray $\ga_j$. The same proof can also be used to choose $P$ to be the $y_1y_2x_3$-plane, with $T_qC$ being $\pi\times \R,$ with $\pi$ a $2$-d special Lagrangian plane.
	\end{rem}
	\begin{proof}We recall that $SU(3)$ acts transitively on the special Lagrangian planes while preserving the special Lagrangian form, with the isotropy group being $SO(3)$ at each individual plane by Section 3 in \cite{HLg}. First, use an element of $SU(3)$ to change the coordinate so that $T_p$ is the $x_1x_2x_3$-plane.  Then use an element of $SO(3)$ to align the nonnegative part of $x_3$ with the intersection of $C$ and $P$. (Here we can also align with the nonnegative of the $x_3$-axis. This is a matter of choice.) Note that $P$ is the product of $x_3$-axis with a $2$-plane $\pi$ in the $x_1x_2y_1y_2y_3$-plane. We claim that the $2$-plane is contained in  $x_1x_2y_1y_2$-plane and is special Lagrangian. Suppose the basis of the $2$-plane has components in $y_3.$ Then any orthonormal basis of $P$ will be of the following form $\pd_{x_3},v=\Si_k a_k \pd_{x_k}+\Si_j \ai_j\pd_{y_j},w=\Si_k b_k x_k+\Si_j \be_j\pd_{y_j}$, with $a_3=b_3=0$, $\ai_3,\be_3\not=0.$ Recall that all the coordinate changes we have used so far lie in $SU(3),$ thus preserving the special Lagrangian form. This implies that $P$ is still calibrated by the special Lagrangian form $\phi.$ Thus, $\phi(v,w,\pd_{x_3})$ is either $1$ or $-1$, since we do not know the orientation of the basis we choose. However, 
		\begin{align*}
			&	|\phi(v,w,\pd_{x_3})|\\=&|\phi(a_1\pd_{x_1},b_2\pd_{x_2},\pd_{x_3})+\phi(a_2\pd_{x_2},b_1\pd_{x_1},\pd_{x_3})\\&+\phi(\ai_1\pd_{y_1},\be_2\pd_{y_2},\pd_{x_3})+\phi(\ai_2\pd_{y_2},\be_1\pd_{y_1},\pd_{x_3})|\\=&|a_1b_2-a_2b_1-\ai_1\be_2+\ai_2\be_1|\\\le&\sqrt{(a_1^2+a_2^2+\ai_1^2+\ai_2^2)(b_1^2+b_2^2+\be_2^2+\be_1^2)}\\
			=&\sqrt{(1-\ai_3^2)(1-\be_3^2)}\\
			<&1,
		\end{align*}which is a contradiction.
		Thus, we deduce that $P=\pi\times x_3$-axis, with $\pi\s \C^2.$ $\pi$ must also be special Lagrangian, since the interior product \begin{align*}
			\pd_{x_3}\res \phi=&\pd_{x_3}\res dx_1dx_2dx_3-\pd_{x_3}\res dy_1dy_2dx_3-\pd_{x_3}\res dy_1dx_3dy_3-\pd_{x_3}\res dx_1dy_2dy_3\\=&dx_1dx_2-dy_1dy_2,
		\end{align*}is precisely special Lagrangian form on $\C^2.$ ($\Re dz_1\w dz_2=\Re(dx_1dx_2-dy_1dy_2+i(dx_1dy_2+dy_1dx_2))=dx_1dx_2-dy_1dy_2.$)
	\end{proof}
	\begin{lem}\label{cocp}
		Suppose we have two pairs of distinct special Lagrangian planes $P_1\cup P_2,$ and $P_3\cup P_4,$ with $P_1\cap P_2$ and $P_3\cap P_4$ both $1$-dimensional. Then there exists $h_1,h_3\in SU(3)$ so that $h_1(P_2)$ and $h_3(P_4)$ both become the $y_1y_2x_3$-plane. Moreover, $h_1(P_1\cap P_2)=h_3(P_3\cap P_4)=x_3$-axis. For some $\rh_j\in\R,$ $h_1(P_1),h_3(P_3)$ can be parametrized as 
		\begin{align}\label{gpm}
			h_j(P_j)=\{(x_1,x_2,x_3,\rh_jx_1,-\rh_jx_2,0)|x_1,x_2,x_3\in\R\}=\gr\na \frac{\rh_j(x_1^2-x_2^2)}{2}.
		\end{align}
	\end{lem}
	\begin{proof}
Several consecutive rotations are used to construct $h_j.$ In order to simplify the presentation, we will not use new notations for the image of $P_l$ under each rotation.		By the remark under \ref{coc}, one can first use $SU(3)$ rotations to make $P_2$ and $P_4$ to align with the $y_1y_2x_3$-plane. Moreover, $P_1$ and $P_3$ are both products of the $x_3$-axis with a special Lagrangian plane in $\C^2$, the $x_1y_1x_2y_2$-plane. Use $P_j'$ to denote the factor of $P_j$ that splits off the $x_3$-axis.  By Section 2.1 in \cite{DJ}, $P_2'$, $P_4'$ are complex planes in the complex coordinate $w_1=x_1+ix_2,w_2=y_1-iy_2$ in the $x_1y_1x_2y_2$-plane. Note that they are distinct from $P_2,P_4$, which coincide with the $w_2$-plane. Thus, $P_1,P_3$ can be parametrized as follows 
		\begin{align*}
			P_j'=\{w_2=e^{2\lam_ji}\rh_jw_1|\rh_jw\in\R_{>0},\lam_j\in \R\}.
		\end{align*}
		Consider the following change of coordinate $\textnormal{diag}[e^{-\lam_j},e^{\lam_j}]$ in $(w_1,w_2)$ coordinate. Then $P_j'$ becomes $\{w_2=\rh_jw_1\}$ in the new parametrization, which corresponds to $\{y_1=\rh_j x_1,y_2=-\rh_j x_2\}$. Direct calculation shows that $\textnormal{diag}[e^{-\lam_j},e^{\lam_j}]$ corresponds to
		\begin{align*}
			\begin{pmatrix}
				\cos\lam_j & \sin\lam_j \\
				-\sin\lam_j & \cos\lam_j
			\end{pmatrix}
		\end{align*}in the original $(x_1+iy_1,x_2+iy_2)$ coordinate. Thus using this rotation in $SO(2)\s SU(2)\s SU(3),$ now we can parametrize $P_1,P_3$ as $P_j'\times x_3$-axis, with
		\begin{align*}
			P_j'=\{(x_1,x_2,\rh_jx_1,-\rh_jx_2)\}.
		\end{align*} Being a product implies invariance along $x_3$-direction. Since the $y_3$-component of $P_j'$ is zero, we deduce the parametrization of $P_j$ in the prompt.
	\end{proof}
	Now, we can prove a much stronger version of Lemma \ref{coc}
	\begin{lem}\label{cocf}
		Suppose we have two pairs of special Lagrangian planes and cones $C_1\cup P_1$ and $C_2\cup P_2,$ with $C_1\cap P_1,C_2\cap P_2$ both being nontrivial unions of rays. For any two selected rays $\ga_1\s C_1\cap P_1,\ga_2\s C_2\cap P_2$, there exist two $SU(3)$ rotations acting on $C_1\cup P_1$ and $C_2\cup P_2,$ respectively, so that after the rotations $P_1$ and $P_2$ become $P,$ the $y_1y_2x_3$-plane, and the tangent plane to any nonzero point on $\ga_1,$ $\ga_2$ of the cones can be parameterized as formula (\ref{gpm}). Moreover, $\ga_1$ and $\ga_2$ are contained in the $x_3$ axis and point in opposite directions.
	\end{lem}
	\begin{proof}
		Apply Lemma \ref{cocp} to the tangent planes of $C_1\cup P_1$ and $C_2\cup P_2,$ at any nonzero point, we get the parametrization (\ref{gpm}). If $\ga_1,\ga_2$ are already of opposite directions, then we are done. If $\ga_1$ coincide with $\ga_2$, use the following $SO(3)$ rotation $\textnormal{diag}[-1,1,-1]$. Note that this preserves both $P$ and the parametrized tangent planes of the form (\ref{gpm}). We are done.
	\end{proof}
	\begin{rem}
		For different cones, even after all the rotations in \ref{cocf}, the parametrized tangent planes of the form (\ref{gpm}) might not coincide with each other, i.e., having different $\rh_j$. When we are creating one line segment of singularities, we are reflecting one fixed cone, so we do not encounter this problem. The added steps in Section \ref{rott} is mainly to address these non-aligned tangent planes.
	\end{rem}
	\subsection{Examples of special Lagrangian cones intersecting a special Lagrangian plane at only 1 ray}\label{exa}
	We will record here the examples of special Lagrangian cones and planes that intersect at only one ray. The cone $C$ is the cone over the torus $F^2=\{|z_1|=|z_2|=|z_2|,z_1z_2z_3=\frac{1}{3\sqrt{3}}\}\s S^5.$ This cone first appeared in Theorem 3.1 in \cite{HLg}. Note that $F^2$ is the orbit of the diagonal subgroup of $SU(3)$ acting on $(\frac{1}{\sqrt{3}},\frac{1}{\sqrt{3}},\frac{1}{\sqrt{3}}).$ Thus, for general calculations it suffices to consider the point $(\frac{1}{\sqrt{3}},\frac{1}{\sqrt{3}},\frac{1}{\sqrt{3}})$ by invariance. The tangent plane $\pi_0$ to that point is generated by $(1,1,1),(i,0,-i),(0,-i,i).$ Apply the real Gram-Schmidt algorithm to this frame in $\R^6$ to get an orthonormal frame in the standard Euclidean metric that spans the same tangent space. (The complex notation might be a little confusing. Note that we are not using the Hermitian inner product and the complex linear Gram-Schmidt.) Abusing the notation, we have
	\begin{align*}
		\pi_0=	\left(
		\begin{array}{ccc}
			\frac{1}{\sqrt{3}} & \frac{1}{\sqrt{3}} & \frac{1}{\sqrt{3}} \\
			\frac{i}{\sqrt{2}} & 0 & -\frac{i}{\sqrt{2}} \\
			\frac{i}{\sqrt{6}} & -i \sqrt{\frac{2}{3}} & \frac{i}{\sqrt{6}} \\
		\end{array}
		\right),
	\end{align*}where each row of $\pi_0$ represents a vector in the orthonormal basis. We can verify that $\pi_0\in SU(3)$ as a matrix. Now, we change $\pi_0$ by multiplying on the left with $$\rh(\tau,\ta)=\left(
	\begin{array}{ccc}
		1 & 0 & 0 \\
		0 & e^{i \ta} \cos (\tau) & e^{i \ta} \sin
		(\tau) \\
		0 & -e^{-i \ta} \sin (\tau) & e^{-i \ta}
		\cos (\tau) \\
	\end{array}
	\right),$$ where $\ta,\tau\in [-\pi,\pi].$
	\begin{defn}\label{planep}
		Define the planes $$P(\tau,\ta)=\rh(\tau,\ta).\pi_0$$ to be generated by rows of the product. 
	\end{defn}Note that $\rh(\tau,\ta)\in SU(3),$ since its lower right block is a composition of an $SO(2)$ rotation and a diagonal $SU(2)$ rotation. 
	\begin{rem}
		\textbf{Warning!} Technically speaking we are using rows to represent vectors, so linear transform should be matrix multiplication on the right. However, we wish to rotate basis vectors $v_j$ into $e^{ia_j}v_j$ with distinct $a_j$ corresponding to distinct $j.$ Thus, we use the relatively strange way of multiplying on the left. 
	\end{rem}	
	\begin{lem}\label{rays}For several special values of $(\tau,\theta),$ we have the corresponding number of intersecting rays of $P(\tau,\ta)$ with $C$ as in the following table. Moreover, all of the intersecting rays formed are not symmetric with respect to the origin by definition of $F.$
		\begin{table}[h]
			\centering
			$\begin{array}{ ccccc }
				\toprule
				\text{Number of intersecting rays} & \text{Value of } (\tau,\ta)\\
				\midrule
				1 & (0,0) \\
				1 & (0,\frac{\pi}{6}),(0,\frac{\pi}{4}),(0,\frac{\pi}{3}) \\
				2 &(\frac{\pi}{4},\frac{\pi}{4}),(\frac{\pi}{4},\frac{\pi}{3})\\
				4 & (0,\frac{\pi}{2}) \\
				\bottomrule
			\end{array}$
		\end{table}
	\end{lem}
	The proof is straightforward calculations and will be verified using Mathematica. The details are left to Appendix \ref{mat}.
	\begin{rem}
		Note that the tangent plane $\pi_0$ (thus $\ta=\tau=0$) only intersects $C$ along one ray. By invariance, this implies that all the other intersection types listed above cannot be tangential to the cone. Moreover, the plane $P(0,\frac{\pi}{2})$ intersects $C$ along a configuration of the four rays from the center of a regular tetrahedron to its four vertices.
	\end{rem}
	\begin{rem}\label{remr}
		Let $v_1,v_2,v_3$ denote the three rows of $\pi_{0}$. Then changing coordinate using $\pi_0\m\in SU(3),$ we can set $v_j=\pd_{x_j}.$ Then $v_1,e^{i\ta}v_2,e^{-i\ta}v_3$ generates $P(0,\ta).$ The coordinate change diag$\{-1,1,-1\}$ (acting on the left) we use in the next section corresponds to sending $v_1$ to $-v_1$, $v_2$ to $v_2$, $v_3$ to $-v_3$ and then extending complex linearly. This actually shows that $P(0,\ta)$ will always be invariant under the change of coordinates we use in the next section. 
	\end{rem}
		\subsection{Remarks on Zhang's gluing constructions in \cite{YZa}}\label{zhang}
	We will make references to Zhang's gluing constructions in the proof of the main theorems. By this we mainly mean the proof of Theorem 4.6 in \cite{YZa}. Zhang's work is very flexible and powerful. He essentially proves the following: suppose we have a calibrated current $T$ in $B_{\e}(0),$ whose singular set is compactly contained in $B_{\frac{\e}{2}}(0),$ then we can glue any smooth compact surface to which transition smoothly to $T$ in the annulus $B_{\e}(0)\backslash B_{\frac{\e}{2}}(0),$ while keeping the newly minted surface still calibrated in a neighborhood, as long as the orientation of $T$ and the glued surface coincides. 
	
	The idea is very simple. The calibration on the newly glued part comes from the natural calibration in the normal bundle of the surface. For the version we use, the detailed proof is contained in Section 2.6 (especially Lemma 2.7) in \cite{ZL1}. 	
	\section{Creating one arc of singularities}\label{ols}
	In this section, we will first do the local argument to produce calibrated minimal surfaces with one arc of singularity (Proposition \ref{oss} below). 
	
	Now suppose $C$ is a special Lagrangian cone and $P$ is a special Lagrangian plane so that $C\cap P$ consists of just one ray $\ga^+$. (For example, we take $
	\tau=0,\ta=\frac{\pi}{4}$ in Lemma \ref{rays}.)
	
	We will do the construction in four steps. The first two steps boil down to gluing $C\cup P$ to a reflect copy to produce the candidate surface. The next two steps produce the calibration form and the metric. 
	
	\begin{figure}[h]
		\centering
		\includegraphics[width=1\linewidth]{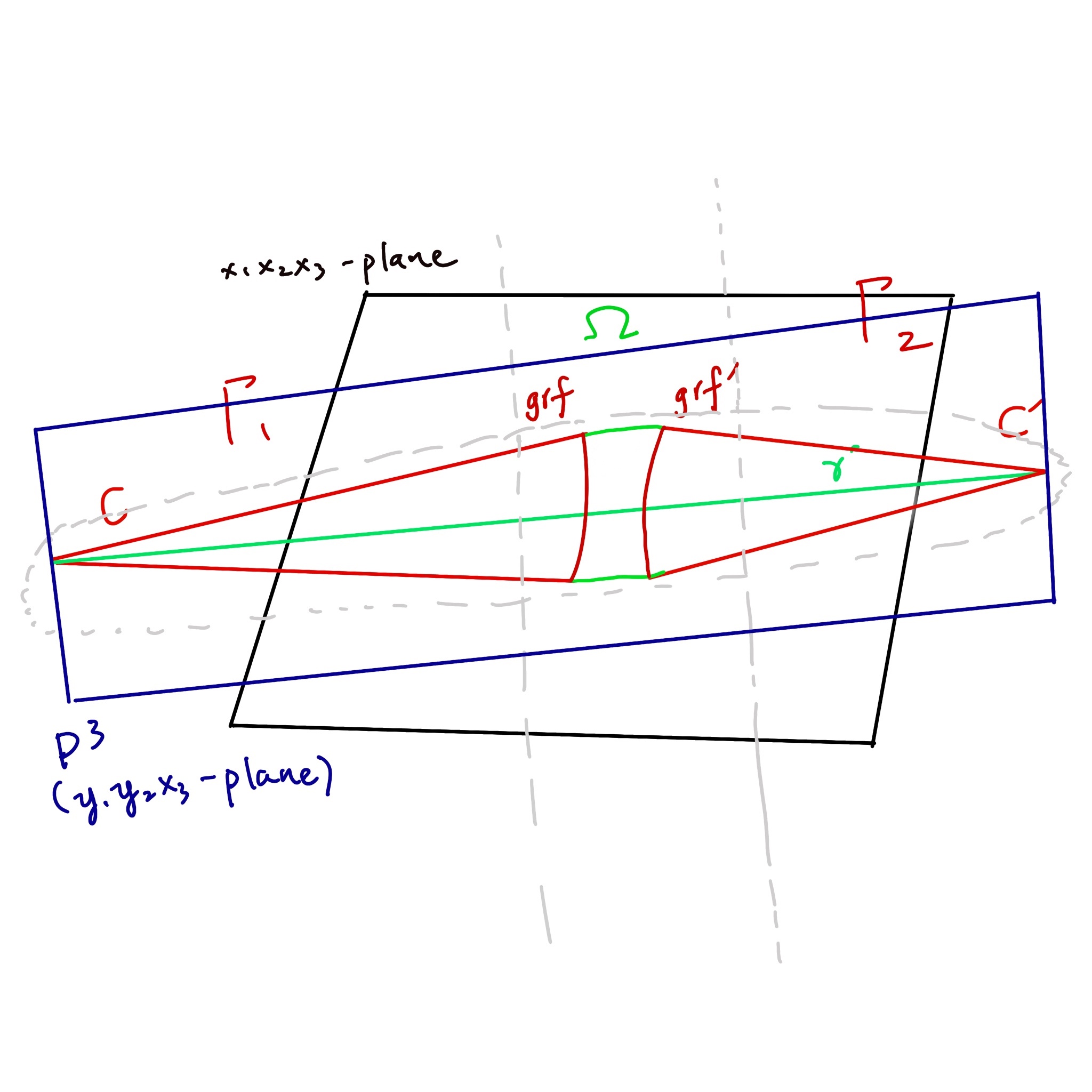}
		\caption[Gluing of the cones]{Gluing of the cone $C$ to its reflected copy $C'$}
		\label{fig:graph-gluing}
	\end{figure}
	\subsection{Step 1  Choice of the preferred coordinate and the reflection}
	Let $q$ be a point on $C\cap P\setminus\{0\}$. We can use Lemma \ref{coc} to change the coordinates to make $T_q C$ the $x_1x_2x_3$-plane and $P=\pi\times x_3$-axis, with $\pi\s \C^2$ special Lagrangian.
	
	We first reflect $C$ along the $y_1$ axis at $0$, i.e., sending $\pd_{y_1}$ to $-\pd_{y_1},$ then $x_1$-axis at $0$, then $y_3$-axis at $0$, and finally along $x_3$-axis at some point $p=(0,0,p_3,0,0,0)$ with $p_3>0$ to get another cone $C'$. Note that the cone point of $C'$ is $2p_3.$ This sequence of reflections preserves the special Lagrangian form since the differential of the affine transform is simply $\textnormal{diag}[-1,1,-1]\in SU(3)$. Let the special Lagrangian current $\llbracket C'\rrbracket $ to be the pushforward of $\llbracket C\rrbracket $ under these reflections. Note that $T_qC$ is naturally invariant under the reflections. $P$ is also invariant under the four reflections, as in Remark \ref{remr}. We deduce that the intersection of $C'$ with $P$ is also a ray contained in the $x_3$-axis. Call it $\ga^-.$ Also, since the sequence of reflections preserves the $x_1x_2x_3$-plane, we have that $C'$ not only intersect the $x_1x_2x_3$-plane but is also tangent to the $x_1x_2x_3$-plane along $\ga^-.$
	
	Moreover, since $\ga$ coincides with the nonnegative part of the $x_3$-axis, after the reflections, $\ga^-$ is the intersection of the $x_3$-axis with $\{x_3\le 2p_3\}.$ Consider the intersecting line segment of $\ga$ and $\ga^-$, the reflected ray. Call this line segment $\ga$. ($\ga$ is the intersection of the $x_3$-axis with the region $\{0\le x_3\le 2p_3\}$.) It starts at the cone point of $C$ and ends at the cone point of $C'.$ The middle point is $p_3,$ the reflection base point. Now, we focus on a tubular neighborhood of $\ga$ in $\R^6.$

 We can suppose the tubular neighborhood $B_r(\ga)$ has a very small radius $r,$ so that in $\{r_0\le x_3\le 2p_3-r_0\}$, $C$ and $C'$ can be represented by two graphs of $f,f':\R^3\to \R^3$ over the $x_1,x_2,x_3$-plane. This follows, e.g., by considering parameterizing $C$ near $r_0$ over its tangent plane at $r_0$ as a graph in the normal bundle and then using $1$-homogeneity to extend the graph. See Remark \ref{rn} for a detailed calculation. Similarly for $C'.$ (Note that in this process simply by taking $r$ small, we can make $r_0$ small as well. Also, if we pull apart the two cones, i.e., making $p_3$ larger, then the same $r$ and $r_0$ work. ) 

Thus, $f$ and $f'$ are natural reflections of each other with respect to $x_3,y_3,x_1,y_1$-axes at $p.$ We use $\Om$ to denote a smaller part of this graphical region $B_r(\ga)\cap\{2r_0\le x_3\le 2p_3-2r_0\}$, (smaller here to transition between the graphs), and the remaining two components of $B_r(\ga)\backslash \Om$ as $\Ga_1$ and $\Ga_2$, with $0\in \Ga_1.$
	\subsection{Step 2 Glue the two cones together}
	We glue $\Ga_1\cap (C\cup P)$ and $\Ga_2\cap(C'\cap P)$ together in $\Om.$ by smooth transition between $f$ and $f'.$ In other words we glue the graph of $f$ and $f'$ by using a cut-off function $\vp$ that depends only on $x_3$. (In other words $\vp$ is $1$ for $x_3\le\frac{3}{2}r_0$ , $0$ for $x_3\ge2p_3-\frac{3}{2}r_0$ and monotonically decreasing in between.) First, by Lemma 2.2 in \cite{HLg}, $f$ can be realized as the gradient of a potential $F$ for $r$ small enough, and thus $f'$ can be realized as the gradient potential of $F'$, the reflection of $F.$ Now we bridge the potentials to be $\ov{F}=\vp F+(1-\vp)F'.$ Note that both $C$ and $C'$ intersects $P$ only along the $x_3$-axis, and have tangents precisely the $x_1,x_2,x_3$-plane. This translates to the fact that $F=0,\na F=0,\textnormal{Hess} F=0$ along $\ga.$ (Here $F=0$ actually is a choice instead of given. We can only deduce $F$ being constant along $\ga.$)  By reflection, the same holds for $F'.$ Now, $$\na\ov{F}=(F-F')\na\vp+\vp(\na F-\na F')+\na F', $$ and $$d\na \ov{F}=(F-F')d\na \vp +\na \vp(dF-dF')+d\vp(\na F-\na F')+\vp(d\na F-d\na F')+d\na F'.$$ This implies that $\na\ov{F}=0$ and $\textnormal{Hess}\ov{F}=0$ along $\ga.$ Thus, the graph of $\ov{F}$ intersects $P$ along $\ga$ and the tangent planes to $P$ at these points are also precisely $x_1x_2x_3$-plane. (It might be the case that $\gr \na \ov{F}$ intersect $P,$ also a graph, at more than $\ga.$ However, since the two have fixed distinct tangents along $\ga$, by making $r_0$ small we can always achieve this.)
	\begin{rem}\label{flat}
		Instead of choosing $F'$ to be a reflected copy of $F$, if one chooses $F'\equiv 0,$ then the above conclusions hold trivially. This implies that instead of gluing two cones together, we can glue the cone to its tangent planes along any ray.
	\end{rem}
	\begin{rem}\label{rn}
		Recall that the transition part of $\vp$ has length $2p_3-3r_0.$ Then the rate the tangent planes to graph$\ov{F}$ change is roughly $C_0\de_0,$ with $C_0$ independent of $p_3,r_0,F,$ and  $$\de_0=(2p_3-3r_0)^{-2}\no{F}_{C^0}+(2p_3-3r_0)^{-1}\no{F}_{C^1}+\no{F}_{C^2}.$$
		Notice that as one goes along a ray in the cone, the cone becomes flatter and flatter in a fixed tubular neighborhood of that ray.  (To see this, first for any point $q$ away from the origin in $\ga$, since the tangent plane is $P,$ if we restrict to a small ball of radius $r$ around that point, we can realize it as a graph of some $f$ over $P.$ Now, since the graph is conical with respect to $0,$ we immediately $f( x/\lam)=\lam\m f(x).$ Thus we have $\no{f}_{C^k(B_r(\lam q))}\le \lam^{-(k-1)} \no{f}_{C^k(B_r(q))} ,$ and thus $\no{f}_{C^k(B_r([q,\lam q]))}\le\no{f}_{C^k(B_r(q))},$ where $[q,\lam q]$ ($\lam>1$) is the line segment from $q$ to $\lam q.$ This is what we mean by the cone getting flatter and flatter along a ray.) Thus, the $\no{F}_{C^k}$ can be made as small as we want if we shrink the radius $r$ of the tube. 
		
		At the same time, if we choose $p_3$ large enough, we will not increase $F$ in $C^0,C^1,C^2.$ Thus with $r$ very small, $p_3$ very large, and $r_0$ relatively small compared to $p_3,$ say $r_0=2^{-5}p_3.$ Then we can make sure $\no{\ov{F}}_{C^k}$ also as small as we want. This makes the bridging part of the surface stays close enough to $x_1x_2x_3$-plane so that the topology of $\Si\cap B_r(\ga)\cap \{2r_0\le x_3\le 2p_3-2r_0\}$ is also trivial, or in other words, the projection of $\Si\cap B_r(\ga)\cap \{2r_0\le x_3\le 2p_3-2r_0\}$ into $x_1x_2x_3$-plane is also a (curved) tube around $x_3$-axis for $\{2r_0\le x_3\le 2p_3-2r_0\}.$
	\end{rem}	\subsection{Step 3 The non-closed prototype of calibration form}
	Now join the graph of $\ov{F},$ with part of the cone $C'$, $C$ in $\Ga_1,\Ga_2$. The transition happens smoothly. Denote this new surface by $\Si$ and the current with the correct orientation $\llbracket \Si\rrbracket .$ This surface $\Si$ is clearly Lagrangian, as we are bridging using Lagrangian graphs. Recall that every Lagrangian plane is the image of $x_1x_2x_3$-plane by a unitary group element $h\in U(\C^3)$, again by \cite{HLg}. 
	
	We claim that for each point $q\in \Si\cap\{r_0\le x_3\le 2p_3-r_0\}$ , we can assign a $h_q\in U(3)$ so that $T_q \Si$ is the image of $\{x_1,x_2,x_3-\text{plane}\}$ under $h_q$, while preserving the orientation. Moreover, we claim that we can make the map $h_q:\Si\to U(3)$ a smooth map.
	
	To see this, first, let $h_q^0:\Si\to \ov{\textnormal{Gr}}_3(\R^6)$ by $ h_q^0=T_q\Si,$ be the map that sends $q$ to the oriented tangent plane $T_q\Si$ in the oriented Grassmannian of $3$-plane in $\R^6$. This map is naturally smooth by the smoothness of $\Si.$ Moreover, the image of $h_q^0$ lies in the smooth submanifold of Lagrangian planes, which is isomorphic to $U(3)/SO(3)$, the orbit of $U(3)$ actions on the oriented $x_1x_2x_3$-plane (1.3 on page 87 of \cite{HLg}). Thus, $h_q^0$ can also be regarded as a smooth map $h_q^0:\Si\to U(3)/SO(3).$ Now consider the projection map $\Pi:U(3)\to U(3)/SO(3).$ This is a submersion and by the constant rank theorem locally we can find a smooth right inverse to this map. Now tangents of $\Si$ are very close to the $x_1x_2x_3$ plane, and are also Lagrangian, thus belonging to a neighborhood of $x_1x_2x_3$-plane in $U(3)/SO(3)$. Let $\Pi\m_0 $ be the right inverse of $\Pi$ in this small neighborhood of $x_1x_2x_3$-plane, i.e., $\Pi\circ\Pi_0\m=\id$ on this neighborhood. Now the map $h_q$ is defined as $h_q=\Pi_0\m\circ h_q^0$, which is smooth since both $\Pi_0\m$ and $h_q^0$ are. By construction, $h_q$ sends $q$ to an $U(3)$ element that rotates the $x_1x_2x_3$-plane to the oriented $T_q\Si.$
	
	Now extend $h_q$ smoothly in value from a smooth function $h_q:\Si\to U(3)$ to $h_q:B_r(\ga)\cap \{r_0\le x_3\le 2p_3-r_0\}\to U(3)$ by requiring that for $q\in B_r(\ga)\cap (\{r_0\le x_3\le \frac{3}{2}r_0\}\cup\{2p_3-\frac{3}{2}r_0\le x\le 2p_3-r_0\})$ and $q\in P$ we have $h\in SU(3).$ To see this, consider the following set of coordinates.
	\begin{lem}\label{ncs}
		Given a smooth function $f:\Om\to \R^3,$ with $\Om\s \R^3$ simply connected, we can find a coordinate system $Q(X_1,X_2,X_3,Y_1,Y_2,Y_3)$ in a neighborhood $O$ of $\gr f\s \R^6$ by setting
		\begin{align*}
			&Q(X_1,X_2,X_3,Y_1,Y_2,Y_3)\\=&(X_1,X_2,X_3,f(X_1,X_2,X_3))+Y_1n_1+Y_2n_2+Y_3n_3,
		\end{align*}where $n_1,n_2,n_3$ is any set of constant vectors so that $\textnormal{span}\{n_1,n_2,n_3,T_{f(x)}\gr f\}=\R^6$ for any $x\in \Om.$
	\end{lem} 
	\begin{rem}
		Since we are dealing with graphs, the choice $n_1=(0,0,0,1,0,0),n_2=(0,0,0,0,1,0),n_3=(0,0,0,0,0,1)$ works for any $f$.
	\end{rem}
	\begin{proof}
		The image of $dQ$ is spanned by $n_1,n_2,n_3,$ and the tangent space of the graph of $f$. Thus with the condition of $n_1,n_2,n_3$ we imposed, this implies that $dQ$ is bijective at any point. By the inverse function theorem, $Q$ is locally a coordinate around every point on $\gr f.$ To show that the coordinates patch together to a large coordinate system that covers all of $\gr f,$ use the argument as in Section 7, 26. Proposition in \cite{BO}. (Note that this book has a strange numbering system, and numbers sit in front of propositions, and theorems, instead of after.) Roughly speaking, the coordinate comes from exponentiating a vector bundle on $\gr f,$ so the proof is similar to the existence of Fermi coordinates.
	\end{proof}
	We already have the choice of $h_q$ on $\Si.$ On $\ga,$ choose $n_1,n_2,\pd_{x_3}$ positively oriented and generating the plane $P,$ and let $n_3$ be the orthogonal complement to $P\cup T_q\Si.$ Thus for $q\in \ga,$ we have $\textnormal{span}\{n_1,n_2,n_3,T_q\Si\}=\textnormal{span}\{P,T_q\Si,n_3\}=\R^6$. Now, $\det(n_1,n_2,n_3,T_q\Si)$ depends smoothly on $q,$ so for any point $q$ in a small neighborhood of $\ga,$ if extend $n_j$ constantly along the $x_1,x_2,x_3$ directions, we must have $\textnormal{span}\{n_1,n_2,n_3,T_q\Si\}=\R^6$. Thus, we can apply Lemma \ref{ncs} to $\Si$ near $\ga.$ Then we simply extend $h$ in constantly along the $Y_j$ directions in the $Q$-coordinates. By construction, $h_q|_P\in SU(3),$ since $h_q|_{\ga}\in SU(3)$, and $P$ is just varying the $Y_1,Y_2,Y_3$ directions of $\ga$ in the $Q$-coordinate. Similarly, we know that $h_q\in SU(3)$ for  $q\in B_r(\ga)\cap (\{r_0\le x_3\le \frac{3}{2}r_0\}\cup\{2p_3-\frac{3}{2}r_0\le x\le 2p_3-r_0\})$.
	
	Let $\phi$ be the special Lagrangian form. We change it pointwise on the tangent space of $\R^6$ to $\ov{\phi}_q=(h_q\m)\du\phi.$ Since we have $h_q\in U(3),$ we see that $\ov{\phi}$ restricted to each tangent space is still a calibration. (Any $SO(n)$ change of calibration on the tangent space still gives a calibration.) Moreover, at each point $q\in \Si,$ $\ov{\phi}$ calibrates $T_q\Si$, since $h_q\m T_q\Si=x_1x_2x_3-$plane and $h_q\m$ preserve the orientation. It also calibrates $P$ at any tangents to $P$ by construction. We set $\det h_q\m=e^{i\ta}$ for some $\ta$ in a small neighborhood of $0$ in $[-\pi,\pi].$ (This is doable since we stay close in a neighborhood of $x_1x_2x_3$-plane.) Since $\phi=\Re dz_1\w dz_2\w dz_3,$ we have
	\begin{align*}
		\ov{\phi}=&\Re (h_q\m)\du dz_1\w dz_2\w dz_3\\=&\Re \det h_q\m dz_1\w dz_2\w dz_3\\=&\Re e^{i\ta}(\phi+iJ\du \phi)\\
		=&\cos\ta \phi-\sin\ta J\du\phi.
	\end{align*}where $J$ is the complex structure matrix on $\C^3.$ Thus, we have
	\begin{align}\label{dta}
		d\ov{\phi}=-d\ta\w (\sin \ta\phi+\cos\ta J\du\phi).
	\end{align}
	\begin{rem}\label{np}
		By construction $\no{\ov{\phi}}_{C^k}=O(\no{{h_q\m}}_{C^k}).$ To estimate the latter, first $\no{h_q}_{C^k}=O(\no{\ov{F}}_{C^{k+2}}),$ since $\gr\na \ov{F}$ is $\Si$ and $h_q$ gives the tangent space to $\Si,$ i.e., controlled by $\textnormal{Hess}\ov{F}.$ Note that $h_q$ is just the identity matrix on $\ga,$ thus with flatter $\ov{F},$ $\no{h_q-\id}_{C^k}$ can be made arbitrarily small. This implies that $\no{h_q\m}_{C^0}$ is controlled by an absolute constant. To see this, note that \begin{align*}
			\no{h_q\m-1}_{C^0}=\no{\frac{\id-h_q}{1-(\id-h_q)}}_{C^0}\le\frac{\no{\id-h_q}_{C^0}}{1-\no{\id-h_q}_{C^0}}. 
		\end{align*} Thus, $\no{h_q\m}_{C^k}=O(\no{h_q}_{C^k}).$ For example, $dh_q\m=-h_q\m (dh_q) h_q\m.$ This gives $\no{h_q\m}_{C^1}=O( \no{h_q}_{C^1}),$ and by induction one can prove that $$\no{h_q\m}_{C^k}=O(\no{h_q\m}_{C^0}^{k+1}\no{h_q}_{C^k})=O(\no{h_q}_{C^k}).$$ To sum it up, we have $\no{\ov{\phi}_q}_{C^k}=O(\no{{h_q}}_{C^k})=O(\no{\ov{F}}_{C^{k+2}}),$ where the constant depends on $r,r_0$ and $k.$ 
	\end{rem}
	\subsection{Step 4 Producing the calibration form}
	We have already constructed the glued surface and a smooth form $\ov{\phi}$ that calibrates both $P$ and $\Si$ on the tangent space level. However, it is not a closed form. We will modify it to be a closed form in this subsection. 
	
	The following lemma says how to get primitives which are zero on a fixed plane with a dimension lower than the form and will be zero on a fixed orthogonal plane if the form itself is. This is basically an application of the homotopy operator in de Rham cohomology. We squash the normal components to the plane using a homotopy. The primitive is zero on the plane because, by projection, we are left only with tangent spaces lower than the dimension of the form. If we restrict to orthogonal planes, which are invariant under homotopy, integration along homotopy is simply integration along the orthogonal lines. Then zero integrates to zero.
	\begin{lem}\label{inte}
		Suppose we have a cylindrical region $\Om\times [-1,1]^k\s \R^{p+k},$ where $\Om\s \R^p$ is a smooth star-shaped region containing $0\in \R^p$. Let $\tau$ be a closed differential $p+l$-form in $\Om\times [-1,1]$, with $l\ge 1$. Then there exists a primitive $\Phi$ of $\tau$ which satisfies $d\Phi=\tau,$ $\Phi|_{\Om\times \{0\}}\equiv 0.$ Moreover, if for some $1\le k_1\le p+k$, $\tau|_P\equiv 0$ on a $k_1$-dimensional plane $P$ with $P\perp \Om,$ (we only require the part of $P$ that doesn't intersect $\Om$ to be orthogonal,) then we also have $\Phi|_P=0.$
	\end{lem}
	\begin{cor}
		With the same assumptions as in Lemma \ref{inte}. If $\ov{\Phi}$ is another primitive of $\tau,$ then $\psi=\ov{\Phi}-\Phi$ is an exact form that differs from $\ov{\Phi}$ in $C^0$ up to $O(\no{\tau}_{C^0}).$
	\end{cor}
	\begin{proof}
		Let $W$ denote the $p$ plane containing $\Om\times\{0\}$ in $\R^{p+k}$. Let $\pi$ be the orthogonal projection of $\R^{p+k}$ onto $W$.  For any $x\in \R^{p+k},$ write $x$ as $x=x^T+x^\perp,$ with $x^T=\pi(x).$ Define a homotopy $G:[0,1]\times \R^{p+k}\to \R^{p+k}$ as follows
		\begin{align*}
			G(t,x)=x^T+tx^\perp.
		\end{align*} 
		It is clear that $G$ is a smooth homotopy between the identity map and $\pi.$
		
		Now recall the following formula for the homotopy invariance in de Rham cohomology \cite{PL}. Let $I$ be the integration operator
		\begin{align}\label{hto}
			I(\tau)=\int_{[0,1]}\pd_t\lrcorner G\du\tau dt.
		\end{align}
		By Cartan's magic formula, we have
		\begin{align*}
			dI(\tau)+Id\tau=\tau-\pi\du\tau.
		\end{align*}
		Since $\tau$ is closed, we have $d\tau=0.$ Note that the range of $d\pi$ is precisely the $\R^p$ tangent to $W$. Thus the $p+l$ form $\pi\du \tau\equiv0.$ This gives 
		\begin{align*}
			dI(\tau)=\tau.
		\end{align*}
		We claim that $\Phi=I(\tau)$ satisfies the conditions as stated. 
		
		Differentiating $G,$ we deduce that $dG_{t,x}(\pd_t)=x^\perp,$ $dG_{t,x}(v)=G(t,v)$ for any $v$ tangent to $\R^{k+p}.$ For any $x\in W,$ we have $dG_{t,x}(\pd_t)=0$, since we have $x^\perp=0.$ This implies $\pd_t\lrcorner G\du(\tau)\equiv 0$ for $x\in W,$ and thus by (\ref{hto}), we have $I(\tau)\equiv 0$ on $W.$
		
		If $P\perp Q,$ then $P$ is invariant by $G.$
		\begin{align*}
			\pd_t\lrcorner (G\du \tau)_{y}(v_2,\cd,v_p)=\tau_{G(t,y)}(dG(\pd_t),\cd,dG(v_p)).
		\end{align*}
		If $\tau|_{P}\equiv0,$ then we clearly have $G(t,y)\in P$ and $\tau_{G(t,y)}=0$ for any $y\in P.$ ($P$ is invariant under the homotopy $G.$) This shows $I(\tau)\equiv 0$ on $P$ if $\tau|_P\equiv0$.
		
		For the proof of the corollary, simply apply the above result.
	\end{proof}
	Now we want to use the above Corollary with $\tau=d\ov{\phi},$ after changing the coordinate to the $Q$-coordinate in Lemma \ref{ncs}. Then there exists a closed $3$-form $\psi,$ which is defined by
	\begin{align}\label{defp}
		\psi=\int_{[0,1]}\pd_t\res G\du\ov{\phi} dt,
	\end{align}where $G(t,x)=(X_1,X_2,X_3,0,0,0)+t(0,0,0,Y_1,Y_2,Y_3)$ in the $Q$-coordinate. We have $\no{\psi-\ov{\phi}}\le C\no{d\ov{\phi}}_{C^0},$ and $\psi\equiv \ov{\phi}$ on $\Si.$ Note that $\no{d\ov{\phi}}_{C^0}\le\no{\ov{\phi}}_{C^1}=O(\no{\ov{F}}_{C^3})$ by Remark \ref{np}. By Remark \ref{rn}, we know that $\no{\ov{F}}_{C^3}$ can be made as small as we want when we shrink $r,$ so we can make $\psi$ as close to $\ov{\phi}$ as we want. To show that $\psi\equiv \ov{\phi}$ on $P,$ we need to show that $d\ov{\phi}|_P\equiv 0.$ This is achieved with the following lemma.
	\begin{lem}\label{ted}
		For the modified special Lagrangian form $\ov{\phi}$ we have
		\begin{align}\label{dte0}
			d\ta|_{P}\equiv 0.
	\end{align} 	\end{lem}
	\begin{proof}
		We have collected some basic facts about immersions of Riemannian submanifolds in the Appendix \ref{bas}. that we use in the following calculations.	It suffices to verify $d\ta=0$ for the basis $\pd_{x_j},\pd_{y_k}$ at every point of $P.$ Recall that $h$ on $P$ is defined by extending constantly in the normal directions of the coordinate system $Q$ in Lemma \ref{ncs}. This immediately implies that $d\ta=0$ along all normal directions of $\Si$ at $\ga,$ i.e., along $\pd_{Y_j}.$ Moreover, $d\ta_q(\pd_{x_j})=d\ta_{\pi_{\Si}(q)}\pd_{x_j},$ where $\pi_\Si$ is the projection to $\Si$ in the coordinate in Lemma \ref{ncs}. Note that for points on $P$, the image of $\pi_\Si$ always lies on $\ga.$ Thus, we only have to verify $d\ta|_{\ga}(\pd_{x_j})=0$ for all $j.$ By III.2.D (2.19) in \cite{HLg}, $d\ta(\pd_{X_j})=0$ for all $j$ along $\ga$ is equivalent to saying that $H_{\Si}=0$ along $\ga.$
		
		Now, recall that the bridging part of $\Si$ is constructed from bridging the Lagrangian potentials $F,F'$, as the graph of $\na(\vp F+(1-\vp)F').$ Thus, the map here is $\ov{f}:\R^3\to \R^{6}$ defined by 
		\begin{align*}
			\ov{f}(x)=(x,\na(\vp (F-F')+F')).
		\end{align*} 	
		Moreover, we have chosen the potential so that $F-F'=0$ along $\ga$. Since the graphs of $F$ and $F'$ are both tangent to $x_1x_2x_3$-plane along $\ga$, we deduce that $dF=dF'=0,\textnormal{Hess}F=\textnormal{Hess}F'=0$ along $\ga.$ The metric coefficients $g_{ij}$ is precisely $\de_{ij}.$ This implies that $v_j=\ov{f}\pf(\pd_j)$ form an orthonormal basis, and the mean curvature of $\Si$ along $\ga$ is precisely 
		\begin{align*}
			H_\Si=&\ri{\na_{v_i}v_i,\pd_{y_j}}\pd_{y_j}\\
			=&\pd_i\pd_i\ov{f}^{y_j}\pd_{y_j}\\
			=&\De((\na\vp)^j (F-F')+\vp(\na(F-F'))^j+(\na F')^j)\pd_{y_j}\\
			=&\bigg\{(F-F')\De(\na\vp)^j+(\na\vp)^j\De(F-F')+2\ri{\na[(\na\vp)^j],\na (F-F')}\\ 
			&+(\na(F-F'))^j\De \vp+\vp\De((\na(F-F'))^j)+2\ri{\na\vp,\na((\na(F-F'))^j)}\\
			&+\De((\na F')^j)\bigg\}\pd_{y_j}\\
			=&\big(\vp\De ((\na(F-F'))^j)+\De((\na F')^j)\big)\pd_{y_j}.
		\end{align*}where the superscript denotes the component.
		
		The same calculation shows that $H_{C}=\big(\De((\na F)^j)\big)\pd_{y_j}$ and $H_{C'}=\big(\De((\na F')^j)\big)\pd_{y_j}.$ This shows that $H_{\Si}=\vp H_C+(1-\vp)H_{C'}=0$ along $\ga,$ and we are done.
	\end{proof}
	Moreover, we have $d\ov{\phi}\equiv0$ in $ \Om\cap (\{r_0\le x_3\le \frac{3}{2}r_0\}\cup\{2p_3-\frac{3}{2}r_0\le x\le 2p_3-r_0\})$, since $h\in SU(3)$ and thus $\ta=0$ in these regions. Thus, integrating over normal parts of $\Si$ in the new coordinate doesn't change the primitive by the above Corollary. This implies the antiderivative $\Phi$ we create for $d\phi$ is zero in $ \Om\cap (\{r_0\le x_3\le \frac{3}{2}r_0\}\cup\{2p_3-\frac{3}{2}r_0\le x\le 2p_3-r_0\})$, and so is $\psi$. To sum it up, $\psi$ is a closed form that coincide with $\phi$ in $ B_r(\ga)\cap (\{r_0\le x_3\le \frac{3}{2}r_0\}\cup\{2p_3-\frac{3}{2}r_0\le x\le 2p_3-r_0\})$ and also on the surfaces $\Si$ and the plane $P.$
	
	Now that we have the form $\psi$ and the glued surface, we will change the metric to make it a calibration. We need to use the following implicit function argument.
	\begin{thm}\label{ift}
		There exists a $C^0$ neighborhood of $\phi$, so that any $3$-form $\tau$ in that neighborhood is the pushforward of $\phi$ under an action in $GL(6)$, i.e., $\tau=h\pf \phi,$ for some $h\in GL(6)$. Moreover, if $\tau$ varies smoothly, then we can choose $h$ to depend smoothly on $\tau.$
	\end{thm}
	\begin{proof}
		The proof here in principle works for some other calibrations as well, so we first use general dimensions. Suppose we have a constant coefficient $k$-form $\phi$ in $\R^n$, i.e., an element of $\bigwedge^k(\R^n)$. Consider the following map $\lambda:GL(\R^n)\to \bigwedge^k(\R^n)$ defined by
		\begin{align*}
			\lam(h)=h\du\phi.
		\end{align*}
		$\lam$ is apparently a smooth map.
		The differential at the identity of this map is
		\begin{align*}
			d\lam(h)(v_1,v_2,\cd,v_k)=\phi(h v_1,v_2,\cd,v_k)+\phi(v_1,h v_2,\cd,v_k)+\phi(v_1,v_2,\cd,hv_k).
		\end{align*}
		Suppose we have $$\dim\ker d\lam=\dim GL(\R^n)-\dim\bigwedge^k(\R^n). $$ Then in a neighborhood of identity, $\lam$ is a submersion. (Put in a local coordinate so that the differential of the first square block is invertible. This is characterized by a nonzero determinant, which gives the neighborhood by smoothness.) Submersions can be locally represented by coordinate projections (Chapter 1, 35. Lemma in \cite{BO}) so there exists a neighborhood $O$ of $\phi$ such that any form $\vp\in O$ can be written as $\vp=h\du \phi$ for some $h\in GL(\R^n)$ close to the identity, and this choice of $h$ depends smoothly on $\vp.$ Moreover, if $\vp=\phi,$ then $h=\id.$ 
		
		Thus, we only need to calculate $\dim \ker d\lam$, the dimension of the Lie algebra of the invariant group of $\phi$, with $\phi$ being the special Lagrangian $3$-form. The details are in the Appendix \ref{cif}
	\end{proof}
	Now in order to use Theorem \ref{ift}, we again take the $Q$-coordinate as in Lemma \ref{ncs}. Recall that $\Si$ is then an open region in the $X_1X_2X_3$-plane, and $P$ is the $Y_1Y_2X_3$-plane. We have a closed form $\psi$ that is $C^0$-close to the special Lagrangian form $\phi.$ Now consider the form $h_q\du \psi$. Recall that $h_q$ is constant along $Y_1,Y_2,Y_3$ directions. Moreover, since $\no{h_q-\id}_{C^0}$ can be as small as we want, we deduce that $\no{h_q\du\psi-\phi}_{C^0}$ can also be made as small as we want. Thus, we can invoke Theorem \ref{ift} to deduce the existence of a smooth function $h_q':B_r(\ga)\to GL(6),$ so that 
	\begin{align*}
		h_q\du\psi=((h_q')\m)\du\phi,
	\end{align*}
	which gives
	\begin{align*}
		\psi=(h_q\m)\du((h_q')\m)\du\phi=((h_q'h_q)\m)\du\phi.
	\end{align*}
	By $\psi\equiv\ov{\phi}$ on $\Si$, $P$, and $B_r(\ga)\cap (\{r_0\le x_3\le \frac{3}{2}r_0\}\cup\{2p_3-\frac{3}{2}r_0\le x\le 2p_3-r_0\})$, we deuce that $h_q\du\psi\equiv \phi$ on these two surfaces and two regions, thus $h_q'$ is identity there.
	
	Now, we just change the metric pointwise to $g$ so that $g_q=((h_q'h_q)\m)\du \de,$ where $\de$ is the standard Euclidean metric. Then $g$ depends smoothly on $p$. Note that by construction the comass of $\psi$ is $1$ in $g$. (Since we change both the special Lagrangian form and the standard metric by the same $GL(6)$ elements.) Thus $\psi$ is a closed calibration form in the new metric $g.$ Moreover, $\Si$ and $P$ are calibrated with respect to $\psi$ in this new metric $g_q$ since by construction $(h_q'h_q)\m\equiv h_q\m$ on $P$ and $\Si$ and $h_q\m$ always sends the tangent plane to the oriented $x_1x_2x_3$-plane, which is calibrated by the special Lagrangian form $\phi.$
	
	Thus, to sum up, we have proven the following.
	\begin{prop}\label{oss}
		Let $\llbracket T\rrbracket =\llbracket P\rrbracket +\llbracket \Si\rrbracket $. Then there exists a small tube $B_r(x_3-\text{axis})$ in $\R^n$ with smooth metric $g_0$, so that  $\llbracket T\rrbracket $ is calibrated by $\psi$ (equation \ref{defp}) in that tube. The singular set of $T$ is the line segment $\ga$. Outside of $\{\frac{3}{2}r_0<x_3<2p_3-\frac{3}{2}r_0\},$ the calibration form $\psi$ coincides with the special Lagrangian form and $g_0$ coincides with $\de$. 
	\end{prop}
	
	\section{Creating a graph of singularities}
	Let $G$ be a connected finite graph. In this section, we will construct current with $G$ as the singular set and calibrate it in a neighborhood. (Proposition \ref{graph}).
	
	The idea is to use the construction in the previous section to realize every edge of the graph. However, we will show that every vertex degree can be realized. 
	\subsection{Every degree can be realized}\label{mexa}
	Notice that in the proof of Proposition \ref{oss}, we have only used the fact that away from the origin and near singular ray, we have only a plane $P$ and a cone $C$ intersecting along that ray. Thus, if we have planes $P_j$ pairwise intersecting $P$ only at the origin, then the picture is still true. This enables us to achieve every degree.
	
	We will modify the examples in Section \ref{exa}. Let $$R(a,b)=\textnormal{diag}\{e^{ia},e^{ib},e^{-i(a+b)}\}$$ denote an element of the diagonal subgroup of $SU(3).$ Recall that we use $\R$-linear row span to denote a plane, so this diagonal action $R$ in the standard basis on $\C^3$ acts on the right in our notation. 
	\begin{prop}\label{anydeg}
		Let $P$ denote the plane $P(0,\frac{\pi}{4})$ (Definition \ref{planep}). Then for any finite number $n,$ there exists exists planes $P_j=P.R(a_j,b_j)$ for $j=1,\cd,n$ with $a_1=b_1=0$ so that each $P_j$ only intersect the Harvey-Lawson cone $C$ along a ray and the pairwise intersection  $P_j\cap P_k$ is always trivial, i.e., $\{0\},$ for $j\not=k$. 
	\end{prop}
	\begin{proof}		
		Given that $P$ intersects $C$ only along a ray, if we rotate $P$ and $C$ simultaneously with $R(a,b),$ then the images $P.R$ and $C.R$ will still intersect along a ray.	Note that the link of $C$ is precisely the orbit of the diagonal subgroup of $SU(2),$ i.e., $R(a,b)$ for all $(a,b)\in\R^2.$ Thus, $C$ is invariant under $R.$ This shows that $P.R(a,b)$ and $C$ only intersect along a ray for all $a,b\in\R.$

		Let $t\in\R$ be a parameter. We will determine whether $P$ and $P.R(t,t)$ intersect nontrivially. Note that it suffices to check whether the system of three complex coefficient linear equations $$\begin{pmatrix}
			c_1&c_2&c_3
		\end{pmatrix}.P=\begin{pmatrix}
			c_4&c_5&c_6
		\end{pmatrix}.P.R(t,t),$$ with $c_j\in\R$ as unknowns, has nontrivial solutions over $\R.$ This system becomes $6$ equations over $\R$ (separating the real and complex parts of the coefficients). Let $$A(t)$$ denote the matrix formed by the coefficients of this system. Then it has nontrivial solutions if and only if $A$ is not injective, i.e., $$\det A=0.$$ 
		
		Note that $\det A(t)$ is a real analytic function (a polynomial in trigonometric functions) for $t\in\R$, so $(\det A)\m(0)$ must be either $\R,$ or a discrete subset of $\R$. Direct calculation shows that for $P=P(0,\frac{\pi}{4})$ (defined in Section \ref{exa}), we have
		\begin{align*}
			A(\frac{\pi}{4})=\left(
			\begin{array}{cccccc}
				\frac{1}{\sqrt{3}} & -\frac{1}{2} & \frac{1}{2 \sqrt{3}} &
				-\frac{1}{\sqrt{6}} & \frac{1}{\sqrt{2}} & 0 \\
				0 & \frac{1}{2} & \frac{1}{2 \sqrt{3}} & -\frac{1}{\sqrt{6}} & 0 &
				-\frac{1}{\sqrt{6}} \\
				\frac{1}{\sqrt{3}} & 0 & -\frac{1}{\sqrt{3}} & -\frac{1}{\sqrt{6}} & 0 & 0
				\\
				0 & 0 & -\frac{1}{\sqrt{3}} & -\frac{1}{\sqrt{6}} & 0 & \sqrt{\frac{2}{3}}
				\\
				\frac{1}{\sqrt{3}} & \frac{1}{2} & \frac{1}{2 \sqrt{3}} & 0 & \frac{1}{2}
				& -\frac{1}{2 \sqrt{3}} \\
				0 & -\frac{1}{2} & \frac{1}{2 \sqrt{3}} & \frac{1}{\sqrt{3}} & \frac{1}{2}
				& \frac{1}{2 \sqrt{3}} \\
			\end{array}
			\right),
		\end{align*} 
		Thus, $\det A(\frac{\pi}{4})=\frac{3024 \sqrt{2}-4752}{15552}\not=0,$ and $(\det A)\m(0)$ must be a discrete set. This implies that $(\det A)\m(0)\cap [-\pi,\pi)$ is a finite set. Let $$m_0=\min_{t\in (\det A)\m(0)\cap [-\pi,\pi),t\not=0}|t|.$$ Then denote
		\begin{align*}
			P_j=P.R(\frac{j-1}{2n}m_0,\frac{j-1}{2n}m_0).
		\end{align*}
		We claim that the planes $P_j$ for $j=1,\cd,n,$ only pairwise intersect at the origin. To determine the intersection of $P.R(\frac{j-1}{2n}m_0,\frac{j-1}{2n}m_0)$ and $P.R(\frac{k-1}{2n}m_0,\frac{k-1}{2n}m_0),$ note that by rotation of $R(-\frac{k-1}{2n}m_0,-\frac{k-1}{2n}m_0),$ it suffices to determine the intersection of $P.R(\frac{j-k}{2n}m_0,\frac{j-k}{2n}m_0)$ and $P.$ Thus, to show that $P_j$ and $P_k$ intersect only at the origin, it suffices to show that $\frac{j-k}{2n}m_0\not\in(\det A)\m(0).$ Since $|\frac{j-k}{2n}m_0|\le \frac{m_0}{2},$ we are done.
	\end{proof}
	\subsection{Every edge can be realized}
	Let $v_k$ denote the vertices of $G$ and $E_{l,k}=E_{k,l}$ denote the edge (if there is such an edge,) between $v_l$ and $v_k.$ For any vertex, let $n_k=\deg v_k$ and $\cu{C^k}+\sum_{j=1}^{n_k}\cu{P_j^k}$ be a realizing collection of cone and planes in Proposition \ref{anydeg}. Note that $P_j^k=P_1^k.R(\frac{j-1}{2n_k}m_k,\frac{j-1}{2n_k}m_k)$ for $m_k>0$. Thus for each pair $C^k\cap P_j^k$ and $C^l\cap P_m^l,$ they coincide up to $SU(3)$-rotations.
	
	Now we place the vertex $v_k$ and the conical point of the corresponding realizing collection $$T_k=\cu{C^k}\res B_{\frac{p_3}{2}}(v_k)+\sum_{j=1}^{n_k}\cu{P_j^k}\res B_{\frac{p_3}{2}}(v_k)$$ along the $x_3$-axis at points $(0,0,2kp_3),$ with $P_3$ as in Proposition \ref{oss}. Let $\ga^k_j$ denote the ray obtained by intersecting $C^k$ and $P^k_j.$
	
	First pick one arbitrary edge $E_{l,m}$ with $l<m$ from $G.$ We will realize $E_{l,m}.$  Pick the ray $\ga_1^l$ and $\ga_1^m$. By construction and Lemma \ref{coc} there exists $\rh_l,\rh_m\in SU(3),$ so that $\rh_l$ ($\rh_m,$ respectively) acting on the standard coordinate centered at $v_l=(0,0,2lp_3)$ (and $v_m,$ respectively), will make $P_1^l$ coincide with $P_1^m$ and the tangent planes to $C^l$ and $C^m$ along $\ga_1^l$ and $\ga_1^m$, away from the cone points, coincide to be the $x_1x_2x_3$-plane. Moreover, the intersecting rays $P_1^l\cap C^l$ and $P_1^m\cap C^m$ are both contained in the positive $x_3$-axis. Now if we rotate $\cu{C^m}+\sum_j\cu{P_j^m}$ with $\textnormal{diag}\{-1,1,-1\}$ (as in Step 1 in the proof of \ref{oss}), then the intersection $C^m\cap P_1^m\cap C^l\cap P_1^l$ will be a line segment from $v_l$ to $v_m.$ Since the other planes only intersect the segment at the cone points, we are in the situation of Proposition \ref{oss}.
	
	However, there are two caveats with the previous reasoning: we carry out two separate incompatible rotations, and the other cones are also sitting on the $x_3$ axis. First, we will resolve these two problems for the first edge. Then we induct to show this for every edge.
	\subsubsection{Carrying out the two rotations simultaneously}
	Recall that the exponential maps on compact Lie groups are surjective  (Corollary 11.10 in \cite{BH}). Thus, there exists $\rh_l^0,\rh_m^0\in\mathfrak{su}(3),$ so that 
	\begin{align*}
		\exp(\rh_l^0)=&\rh_l,\\
		\exp(\rh_m^0)=&\rh_m.\textnormal{diag}\{-1,1,-1\}.
	\end{align*} 
	Now let $b:\R\to\R$ be a non-negative function that is 1 on $[-\frac{p_3^2}{4},\frac{p_3^2}{4}]$, zero on $(-\infty,-\frac{9}{16}p_3^2]\cup[\frac{9}{16}p_3^2,\infty),$ and non-zero on $(-\frac{9}{16}p_3^2,-\frac{1}{4}p_3^2)\cup(\frac{1}{4}p_3^2,\frac{9}{16}p_3^2).$ Consider the smooth vector field $u^{l,m}$
	\begin{align}\label{rot}
		u^{l,m}(x)=b(|x-v_l|^2)[(x-v_l).\rh_l^0]+b(|x-v_m|^2)[(x-v_m).\rh_m^0],
	\end{align}
	and let $U_t^{l,m}(x)$ be the one-parameter family of diffeomorphisms associated with $u^{l,m}.$
	By construction, $u^{l,m}$ is supported in $B_{\frac{3}{4}p_3}(v_l)\cup B_{\frac{3}{4}p_3}(v_m).$
	\begin{figure}[h]
		\centering
		\includegraphics[width=0.7\linewidth]{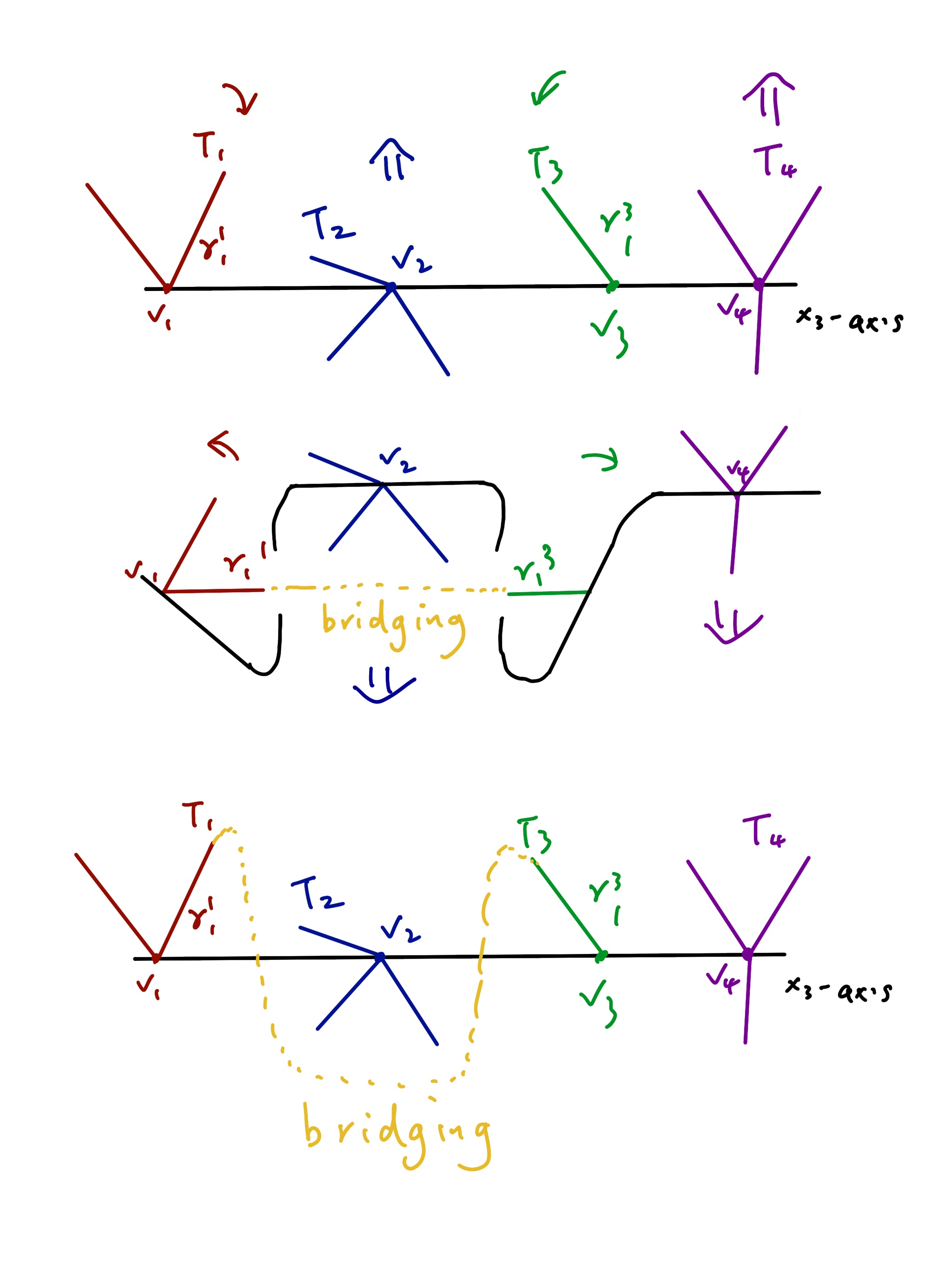}
		\caption{Pushing other vertices away to make space for forming the edges}
		\label{fig:pushing-things-around}
	\end{figure}
	On $B_{\frac{1}{2}p_3}(v_l)$, we have $u^{l,m}(x)=(x-v_l).\rh_l^0$. The vector field $(x-v_l).\rh_l^0$ integrates into the one parameter family of rotations $\exp(t\rh_l^0)$ around $v_l.$
	Now differentiating $1=\ri{x.e^{t\rh^0},x.e^{t\rh^0}}$ for $\rh_0\in\mathfrak{su}(3)$, we know that $\ri{x.\rh^0,x}=0.$ Thus, $u^{l,m}(x)$ is always tangent to spheres centered around $v_l$ or $v_k$ and $V_t$ will always preserve these spheres. Thus, $U^{l,m}_t(x)$ is  the rotation $\exp(t\rh_l^0)$ around $v_l$ on $B_{\frac{1}{2}p_3}(v_l).$ Same reasoning shows that $U^{l,m}_t(x)$ is the rotation $\exp(t\rh_m^0)$ on $B_{\frac{1}{2}p_3}(v_m).$ 
	
	To sum it up, the diffeomorphism $U^{l,m}_1$ coincides with the identity map outside of  $B_{\frac{3}{4}p_3}(v_l)\cup B_{\frac{3}{4}p_3}(v_m).$ And it rotates $T_l$ and $T_m$, so that  $P_1^l=P_1^m$ and $T_q C^l=T_{q'}C^l=x_1x_2x_3$-plane for $q\in \ga^l_1\setminus\{0\},q'\in \ga_1^m\setminus\{0\}$. Moreover, $\ga_1^l\res B_{\frac{p_3}{2}(v_m)}$ points towards the positive $x_3$-direction and $\ga_1^m$ points towards the negative $x_3$-axis. Also, since all the realizing collections stay in a ball of radius $\frac{p_3}{2}$ and are at least $2p_3$ away from each other, all the other realizing collections stay unchanged under $U^{l,m}_1.$
	\subsubsection{Pushing other vertices away}
	Now let $\nu_{l,m}$ be a unit vector orthogonal to the $x_3$-axis. Take the slice of $\R^2$ in $\C^3,$ spanned by the $x_3$ axis and $\nu_{l,m}$. Call this slice book page $\{l,m\}$. 
	
	Define the one parameter of diffeomorphisms $W_t$ as generated by the vector field 
	\begin{align}\label{push}
		w^{l,m}(x)=(1-b(|x-v_l|^2))(1-b(|x-v_m|^2))\nu_{l,m}
	\end{align}
	The effect of $W_t^{l,m}$ on $\C^3$ is simply pushing points outside of $B_{\frac{3}{4}p_3}(v_l)\cup B_{\frac{3}{4}p_3}(v_m)$ into the positive $\nu_{l,m}$-direction, thus forcing all the other realizing collections to leave a $t$ neighborhood of the $x_3$-axis.
	
	Now apply $W_{2p_3}^{l,m}\circ U_1^{l,m}.$ As the effect, we have moved every realizing collection outside of $T_l$ and $T_m$ into the positive $\nu_{l,m}$ direction by $2p_3,$ and also have rotated the realizing $T_l$ and $T_m$ into the right position for gluing along the rays $\ga_1^l$ and $\ga^m$. Now we prolong the truncated $C^m\res B_{\frac{1}{2}p_3}(v_l)$ into $C^m$ and $P^m_1\res B_{\frac{1}{2}p_3}(v_l)$ into $P^m_1,$ and similarly for $v_m,$ both around a small tubular neighborhood around the $x_3$-axis, so that the prolonged cones and planes never intersect the other realizing planes in this tube. This is doable since, for each vertex, the planes all intersect trivially at the origin. Then we apply Proposition \ref{oss} with $r_0=\frac{1}{2}p_3$ to glue $C^m$ to $C^l$ while keeping everything calibrated by a form $\psi_{m,l}$ in a metric $g_{m,l}$ which coincides with the standard special Lagrangian form $\phi$ and the standard metric around a ball of radius $\frac{p_3}{2}$ around $v_k$ and $v_l.$ Finally we apply $(U_1^{l,m})\m \circ (W_{2p_3}^{l,m})\m,$ and pullback the form $\psi_{m,l}$ and the metric $g_{m,l}$ to $((U_1^{l,m})\m\circ (W_1^{l,m}))\m B_{r_{l,m}}([v_l,v_m]).$  This gives a calibrated current bridging $C^m\cup P^m_1$ to $C^l\cup P^l_1$ and the singular set together forming the $E_{l,m}.$ 
	
	Moreover, note that this curved $E_{l,m}$ stays in the book page $\{l,m\}$ and is at least $2p_3$ away from $x_3$-axis on $(B_{p_3/2}(v_l)\cup B_{p_3/2}(v_m))\cp$ and at least $\frac{3}{2}p_3$ away from all the other vertices. Thus if we shrink the gluing radius, we can ensure that the bridging current also stays at least $p_3$ away from $x_3$-axis on $(B_{p_3/2}(v_l)\cup B_{p_3/2}(v_m))\m$ and always at least $p_3$ away from all the other vertices.
	\begin{figure}[h]
		\centering
		\includegraphics[width=0.7\linewidth]{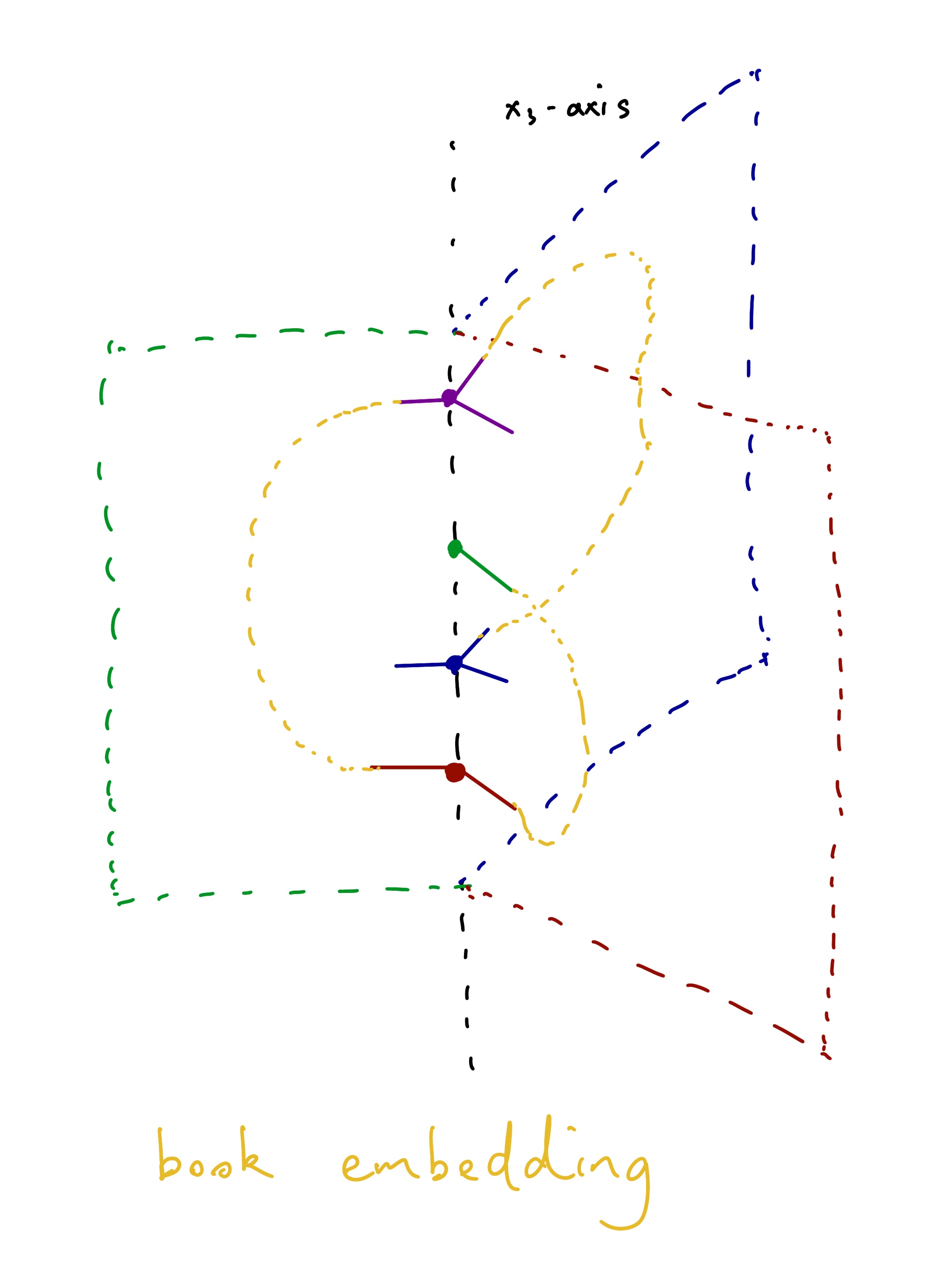}
		\caption{The book embedding of a graph (\textbf{Warning:} this is only a schematic representation up to twisted diffeomorphisms)}
		\label{fig:book-embedding}
	\end{figure}
	\subsubsection{Induction for all edges}
	Now we want to repeat this process for every edge. We do it by induction. Suppose we have already constructed $k$ different edges, which do not intersect along the bridging part. Moreover, suppose both $E_{l,m}$ and the bridging current stays at least $p_3$ away from $x_3$-axis on $(B_{3p_3/4}(v_l)\cup B_{3p_3/4}(v_m))\cp$ and always at least $p_3$ away from all the other vertices. (For our first edge, note that the diffeomorphism $U_1^{l,m}$ is identity when at least $3p_3/4$ away from $v_l$ and $v_m.$ The diffeomorphism $W_{2p_3}^{l,m}$ pushes everything at least $3p_3/4$ away, into the $\nu_{l,m}$ direction by $2p_3$.) Now suppose we want to construct a new edge $E_{a,b}.$ Again, using the elements of Lie algebra, we can construct a smooth diffeomorphism $U^{a,b}_1$ generated by smooth vector fields $u^{a,b}$ as in (\ref{rot}), that rotates the realizing collection $T_a,T_b$ into position ready for gluing. Suppose the two truncated rays for gluing are $\ga^a_1,\ga^b_1.$
	
	Note that in this process, it is possible that some edge emitting from $v_a,v_b$ might cross the $x_3$-axis, on $(B_{p_3/2}(v_a)\cup B_{p_3/2}(v_b))\cp$ (By induction hypothesis the bridging current emitting from other vertices will stay at least $p_3$ away from $v_a,v_b,$ thus unchanged under $U^{a,b}_1.$) 
	
	We now need to choose a unit normal vectors $\nu_{a,b}$ to the $x_3$-axis, in order to span the book page $\{a,b\}$ (slice of $\R^2$) together with the $x_3$-axis. 
	\begin{lem}
		There exists an $\nu_{a,b}$, so that all the $T_j$ and bridging parts only intersect book page $\{a,b\}$ along the $x_3$-axis.
	\end{lem}
	\begin{proof}
		Let $\pi_3$ be the projection into the orthogonal complement of $x_3$-axis. Then let $\pi_{5}$ be the projection from $\R^5\setminus\{0\}$ to the $4$-dimensional real projective space $RP^4,$ i.e., modulo scalar multiplication.  Choosing $\nu_{a,b}$ up to the equivalence of its span with $x_3$-axis, amounts to choosing a point $w\in RP^4.$ Let $S$ denote the support of all $T_j$ and the bridging parts. Note that $S$ is a union of 3-dimensional smooth manifolds with boundary and cone parts. Since $\pi_5\circ\pi_3$ is a locally Lipschitz map, the Hausdorff dimension of $\pi_5\circ\pi_3(S)$ is at most a $3.$ Thus, $\pi_5\circ\pi_3(S)\cp$ is non empty	. For any point $x$ in the complement of $\pi_5\circ\pi_3(S)$, the corresponding plane spanned by the $x_3$-axis and $x$ cannot intersect $S$ outside of the $x_3$-axis. Thus, we can take any point in $\pi_5\m( x)\cap S^4.$ 
	\end{proof}
	Choose $\nu_{a,b}$ as in the above lemma. Then since the bridging parts coming from other edges are closed and do not intersect $\ga_1^a,\ga_1^b,$ we deduce that their intersections with the book page are at least some small distance, say $\e_{a,b},$ away from $\ga_1^a,\ga_1^b,$ so that we can push everything in between away. Again we use a diffeomorphism $W_t$ generated by vector fields like (\ref{push}). However, the auxiliary function $b$ has to be modified, with value 1 on $[-\frac{p_3^2}{4},\frac{p_3^2}{4}]$, zero on $(-\infty,-\frac{(p_3+\e_{a,b})^2}{4}]\cup[\frac{(p_3+\e_{a,b})^2}{4},\infty),$ and non-zero on $(-\frac{(p_3+\e_{a,b})^2}{4},-\frac{p_3^2}{4})\cup(\frac{p_3^2}{4},\frac{(p_3+\e_{a,b})^2}{4}).$ Moreover, the time with $t_a$ is large enough, so that on the $x_3$-axis in the page $\{a,b\}$, everything in $(B_{p_3/2+\e_{a,b}}(v_a)\cup B_{p_3/2+\e_{a,b}}(v_b))\cp$ is pushed into the $\nu_{a,b}$ for at least $2p_3.$ Now we do the gluing construction as in Proposition \ref{oss}, and then pullback via $(W^{a,b}_{t_a}\circ U^{a,b}_1)\m$ to get the twisted bridging current in the page $\{a,b\}$ and satisfies all the induction hypothesis.
	\begin{rem}\label{bdry}
		Recall that in the gluing process, only two things happen: two planes are glued together and two cones are glued together. Note that $G$ is connected in this section, so all cones are glued together with multiple bridges. We call this glued-together conical surface $\Si.$ The planes are only glued together pair by pair. Call these glued-together pairs of planes $Q_j.$ Moreover, by shrinking the radius $r$ of $B_r(G)$ in the Euclidean metric, we can ensure that $\Si$ and $Q_j$ are calibrated in $B_r(G)$ with respect to the new metric and calibration form. By Remark \ref{rn}, we can assume that $\pd \Si$ and $\pd Q_j$ are smooth surfaces. 
	\end{rem}
	
	To sum it up, we have proven the following.
	\begin{prop}\label{graph}
		For any finite graph $G,$ there exists an embedding of $G$ into $\R^6,$ so that there exists a smooth metric $g$ and a calibration form $\psi$ with respect to $g,$ both defined in a neighborhood of $G.$ And $G$ is the singular set of a $\psi-$calibrated current in metric $g,$ \begin{align*}
			T=\cu{\Si}\res B_r(G)+\sum_l \cu{Q_l}\res B_r(G),
		\end{align*}
		where $T$ coincides with a special Lagrangian current near each of its vertices. Each $Q_l$ is smooth and $\Si$ has isolated singularities modeled on special Lagrangian cones.
	\end{prop}
	\section{	Proof of Theorem \ref{anyg}}
	Recall that in all previous sections (Proposition \ref{oss} and \ref{graph}), the surfaces $T=\cu{\Si}\res B_r(G)+\cu{Q_l}\res B_r(G)$ that we consider stay in a tubular neighborhood $B_r(G)$ of $G,$ and have boundaries in $\pd B_r(G).$ Moreover, $G$ is connected. To prove Theorem \ref{anyg}, we need to make the surfaces complete, connect all possible components and glue them to a homology class on the target manifold $M.$
	\subsection{Making the surfaces complete and connected}\label{add}
	Let us first deal with $G$ connected. Notice that $\pd \Si$ and $\pd Q_l$ are all embedded orientable surfaces and the surfaces are smooth near the boundaries. Thus, one can always find embedded orientable $3$-manifolds $\Si'$, $Q_l'$ so that $\pd \Si'=\pd \Si$ and $\pd Q_l'=\pd Q_l$ and the primed notation ones extend the original ones smoothly along their common boundary. Note that we can decompose $G$ into unions of line segments, i.e., $1$-d manifolds. Thus, by Theorem 4.5.6 (Transversality Theorem) in \cite{CWd}, we deduce that after a generic perturbation, we can assume that the current $$T'=\cu{\Si'}-\cu{\Si}+\sum_l \cu{Q_l}-\cu{Q_l'}$$ has only $G$ and some isolated points $q_1,\cd,q_n$ as the singular set. 	Near $G.$ $T'$ coincides with $T.$ Moreover, around any point $q_j,$ $T$ decomposes into the sum of two oriented disks $D_1+D_2$. 
	
	Now, apply Lemma \ref{ncs} to $D_1$, and arbitrary sections in the normal bundle of $D_1$. We get a coordinate system with $D_1$ being the $x_1x_2x_3$ coordinate plane, and $D_2$ a surface transversal to $D_1.$ Then we apply Lemma \ref{ncs} to $D_2$ in this coordinate, with $\pd_1,\cd,\pd_3$ being the vector fields. In this new coordinate and restricting to smaller disks, we can assume $D_1+D_2$ is just the sum of radius $\e$ oriented disks in the $x_1x_2x_3$ and $y_1y_2y_3$ planes for some $\e>0$. Now we can replace $D_1+D_2$ with a smooth manifold $K$ so that $\pd K=\pd D_1+\pd D_2$ and $K\s D_1\times D_2$ (For example, consider $K$ parameterized by $K=\{(\e e^{it}\om,\pm \e e^{it}\om)|\om\in S^1,t\in [0,\frac{\pi}{2}]\}$. The $\pm$ is to account for orientations.) Replace $D_1+D_2$ with $K$ and then smoothly make the transition along the boundary.
	
	As a result we get a boundaryless current $T''$ whose singular set is $G$ and near $G$ coincides with $T.$ 
	
	For $G$ with several connected components $G_1,\cd,G_n,$ carry out the construction above and embed each result in a ball of radius large enough in $\R^6.$ No do a connected sum of the resulting $T_1'',T_2'',\cd$ on the smooth parts to get $T''=T_1''\#\cd\# T_n''.$ Then again $T''$ is boundaryless and has $G$ as the singular set and near $G$, and $T''$ is calibrated in a smooth metric by a smooth form $\phi.$ 
	\subsection{Gluing via Zhang's constructions}
First of all, the reader is suggested to read Section \ref{zhang}.

	Since $b_3\not=0,$ by Corollaire II.30 in \cite{RT}, there exist a nontrivial $3$-d homology class $[N_0]$, so that it has  a smoothly embedded oriented representative $N_0.$ 
	
	Note that the $T''$ in the previous section lies in a compact set in $\R^6,$ so we can embed it into the standard $6$-sphere $S^6$.
	
	Now we can do a simultaneous connected sum of $N_0$ with (a neck to the smooth part of) $T''$ and $M$ with $S^6.$ Note that $N\# S^6$ is diffeomorphic to $N$ and $[N_0\# T'']=[N_0]$ as homology classes by homotopy invariance.
	
	Moreover, we have a smooth calibration and a smooth metric around the singular set of $N_0\# T''.$ Now we want to apply Zhang's construction (Section \label{Zhang}) to get a global metric and a calibration form. Roughly speaking, we first glue the calibration of the singular part to a calibration in the normal bundle on the smooth part. Then we make the metric away from $N_0\# T''$ large to make minimizers lie near $N_0\# T''$. Finally, we find a closed form $\psi$ with $N_0\# T''(\psi)=1$ and transit the local calibration to $\psi.$ The detailed proof is contained in Lemma 2.7 in \cite{ZL1}.

For the assumptions Lemma 2.7 in \cite{ZL1}, most conditions are easy to check. We will just mention two that take some effort. Condition 1.c.i follows from Lemma 2.10 in \cite{ZL1}. It is straightforward to check that our $N_0\# T''$ admits a Whitney stratification. By Lemma 2.8 in \cite{ZL1}, Condition 2.e.i is satisfied. Thus, Lemma 2.7 in \cite{ZL1} can indeed be applied here and we are done.
	\section{Gluing of pairs of cones and planes with different angles}\label{rott}	This section is not used in the proof of the main theorems but contains an enhancement of the construction in the previous sections.
	
	The following question has been raised to the author by many people, including but not limited to Professor Max Engelstein, Professor Nick Edelen, and Professor Yongsheng Zhang:
	
	\textit{"Given an embedded finite graph $G$ with straight line edges in $\R^6,$ is it possible to carry out the above construction to make $G$ the singular set?"} 
	
	First of all, one needs to have a collection of special Lagrangian cones and planes realizing each vertex. It turns out this is sufficient to give a positive answer to the above question.
	
	However, general pairs can have intersection angles that are not equal, i.e., not reflections of each other. Thus, the gluing construction in the previous sections needs modifications. We will show how to modify the constructions and later this will be used in an extension of \cite{ZL}.
	
	Suppose we have two sets of special Lagrangian cones $C_1,C_2,$ and special Lagrangian plans $P_1,P_2.$ We assume that $C_l\cap P_l$ consists of several different rays only and the intersection along each ray is non-tangentially. Now pick a ray $\ga_1$ from $C_1\cap P_1$ and $\ga_2$ from $C_2\cap P_2.$ Our aim in this section is to prove the following.
	\begin{prop}\label{oe1}
		Using translations and simultaneous $SU(3)$ rotations on each pair $C_l\cup P_l$, we can place $C_1\cup P_1$ and $C_2\cup P_2$ into $\R^6$ so that $P_1=P_2.$ The vertex of $C_1$ is at $0,$ and the vertex at $C_2$ is at $(0,0,2p_3)\s \C^3$ with arbitrary $p_3>0$. $\ga_1$ is the nonnegative $x_3$-axis, while $\ga_2$ is contained in the $x_3$-axis and pointing towards the nonpositive $x_3$ axis. Denote $\ga=\ga_1\cap \ga_2.$ We have the following.
		\begin{itemize}
			\item in a tubular neighborhood $B_r(\ga),$ we can glue $C_1$ to $C_2$ smoothly to a current $C_{12}$;
			\item the glued together $C_{12}$ is singular only at the conical points of $C_1,C_2$ and the singular set of $T=C_{12}+P_1$ is precisely $\ga$ union with all the other rays in $C_1\cap P_1,C_2\cap P_2$ (suitably truncated) except for $\ga_1,\ga_2;$
			\item there exists a smooth differential form $\psi$ in a tubular neighborhood $B_r(\ga),$ which is a calibration form with respect to a smooth metric $g;$
			\item $\psi$ calibrates $C_{12}+P_1$. 
		\end{itemize}
	\end{prop}
	\subsection{Proof of Proposition \ref{oe1}}
	Pick any edge on the graph $G,$ connecting two vertices $v_1,v_2.$ Suppose we have two pairs $C_1\cup P_1$ realizing $v_1$ and $C_2\cup P_2$ realizing $v_2.$ Fix any choice of rays $\ga_1\s C_1\cap P_1$ and $\ga_2\s C_2\cap P_2$ of the two realizing pairs of $v_1,v_2$. 
	
	Recall the construction of the calibration and surface in the previous section. Roughly speaking, the steps are as follows.
	
	Step 1, choice of the preferred coordinate and the reflection;\\
	Step 2, glue the two cones together;\\
	Step 3, the non-closed prototype of the calibration;\\
	Step 4, producing the calibration form;
	
	\subsection*{Step 1}
	First, we need to prescribe the positions of the cones. Recall that $\ga_1\s C_1\cap P_1$ and $\ga_2\s C_2\cap P_2$ are the two rays we choose from the realizing pairs. Apply Lemma \ref{cocf}. We get two pairs $C_1\cup P$ and $C_2\cup P,$ with $\ga_1$ being the non-negative $x_3$-axis and $\ga_2$ being the non-positive $x_3$-axis. Moreover, for any point on $\ga_j$ except for $0$, the tangent plane to $C_j$ can be parametrized using (\ref{gpm}). We now translate $C_2\cup P$ along the $x_3$-axis by $(0,0,2p_3),$ with $p_3>0$ a fixed real number. Now $\ga_1$ and the translated $\ga_2$ intersect along a line segment $\ga.$ The point $p=(0,0,p_3)$ is the midpoint of $\ga.$ Abusing the notation a little, we will denote the tangent plane to any point of $C_j$ on $\ga_j$ by $TC_j$.
	
	If $\rh_1=\rh_2$ in the parametrization of $TC_j$, then exactly the same argument as in the previous section will glue the two realizing pairs together. From now on, we assume that $\rh_1\not=\rh_2.$
	\begin{figure}[h]
		\raggedleft 
		\includegraphics[width=0.8\linewidth,angle=90]{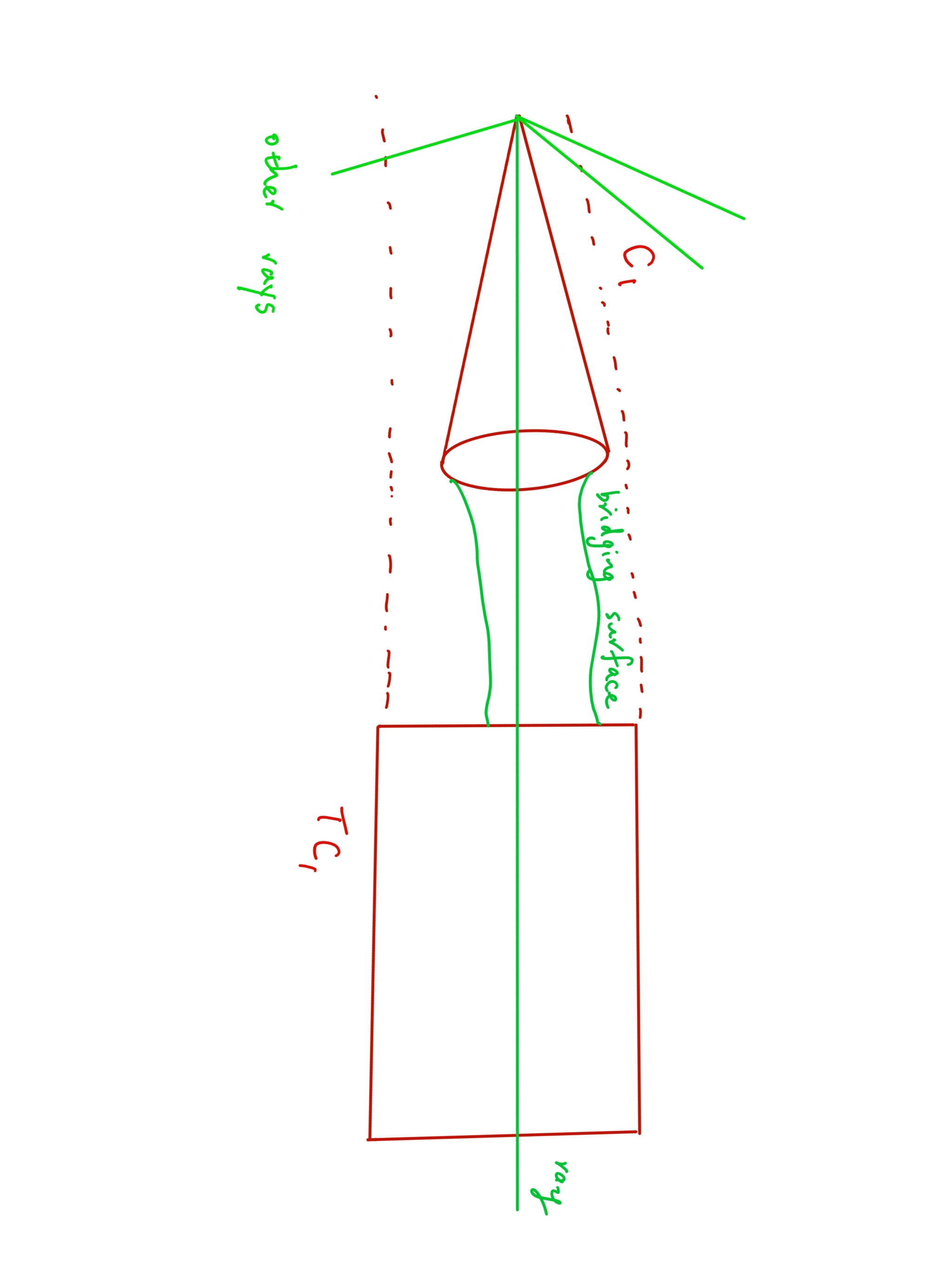}
		\caption{Step 2.1 Gluing tangent planes to the cones}
		\label{fig:gluing-tangents}
	\end{figure}
	\subsection*{Step 2.1}
	Instead of gluing between $C_j$ directly, we first glue $TC_j$ to $C_j.$ To be precise, we take a very small neighborhood of $B_r(\ga)$ of $\ga$ so that for some fixed $0<r_0<\frac{1}{5},$ the cone $C_1$ (respectively, $C_2$) is graphical over $TC_1$ (respectively, $TC_2$) for $\{\frac{1}{5}r_0<x_3<\frac{2}{5}r_0\}\cap B_r(\ga)$ (respectively, $\{2p_3-\frac{2}{5}r_0<x_3<2p_3-\frac{1}{5}r_0\}\cap B_r(\ga)$). Then, by Lemma \ref{flat}, using the argument in Section \ref{ols}, we can glue $TC_1$ to $C_1$ using a Lagrangian graph over $TC_1$ in $ \{\frac{2}{5}r_0<x_3<\frac{4}{5}r_0\}\cap B_r(\ga)$ to get a Lagrangian surface $C'_1,$ which coincide with $TC_1$ in $\{x_3>\frac{4}{5}r_0\}\cap B_r(\ga)$. Moreover, $C'_1$ is calibrated with respect to a metric $\de_1$ and a form $\om_1$ which coincides with the standard Euclidean metric and special Lagrangian form except on $ \{\frac{2}{5}r_0<x_3<\frac{4}{5}r_0\}\cap B_r(\ga)$.
	
	Now we do the same argument to glue $TC_2$ to $C_2.$ We will get a Lagrangian graph bridging $C_2$ to $TC_2$ in $ 2p_3-\{\frac{4}{5}r_0<x_3<2p_3-\frac{3}{5}r_0\}\cap B_r(\ga)$. The glued together Lagrangian surface $C'_2$ coincides with $TC_2$ in $\{x_3<2p_3-\frac{4}{5}r_0\}\cap B_r(\ga)$. Moreover, $C'_2$ is calibrated with respect to a metric $\de_2$ and a form $\om_2$ which coincides with the standard Euclidean metric and special Lagrangian form except on $\{\frac{2}{5}r_0<x_3<\frac{4}{5}r_0\}\cap B_r(\ga)$.
	\begin{figure}[h]
		\centering
		\includegraphics[width=1.1\linewidth]{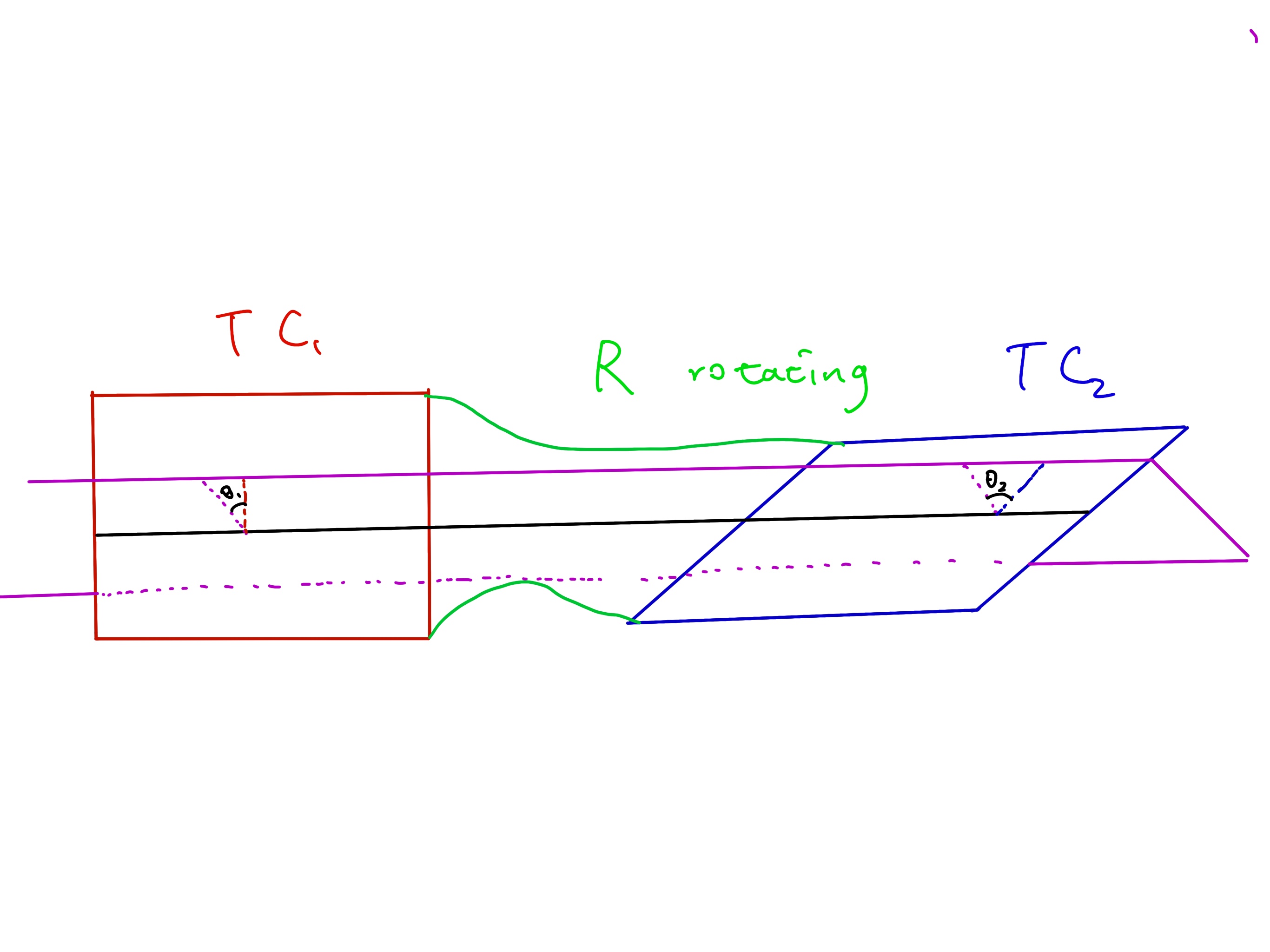}
		\caption{Step 2.2 Bridging using rotating tangents}
		\label{fig:rotating-tangent}
	\end{figure}
	\subsection*{Step 2.2}
	We now want to glue $C'_1$ to $C'_2$ in $B_r(\ga)\cap \{\frac{3}{2}r_0<x_3<2p_3-\frac{3}{2}r_0\}.$ Note $C_1'$ and $C_2'$ is just $TC_1$ and $TC_2$ in $\{r_0<x_3<2p_0-x_3\}$, which can be parametrized as
	\begin{align}
		TC_j=\{(x_1,x_2,x_3,\rh_jx_1,-\rh_jx_2,0)|x_1,x_2,x_3\in\R\}=\gr\na \frac{\rh_j(x_1^2-x_2^2)}{2}.
	\end{align}
	Consider the bridging surface Lagrangian graph $R$ defined on $x_1x_2x_3-$plane$\cap B_r(\ga)\cap \{r_0<x_3<2p_3-r_0\}$ by
	\begin{align*}
		R=\gr\na\frac{\rh(x_3)(x_1^2-x_2^2)}{2},
	\end{align*} 
	with $\rh_3$ a monotonic function, and $\rh(x_3)\equiv \rh_1$ for $x_3<\frac{3}{2}r_0$ and $\rh(x_2)\equiv \rh_2$ for $x_3>2p_3-\frac{3}{2}r_0$.
	
	Now, use $L$ to denote $\na\frac{x_1^2-x_2^2}{2},$ which is just $(x_1,-x_2,0).$ We have
	\begin{align*}
		\na \frac{\rh(x_3)(x_1^2-x_2^2)}{2}&=\rh(x_3)L+\frac{x_1^2-x_2^2}{2}\na \rh(x_3)=(\rh(x_3) x_1,-\rh(x_3) x_2,\rh'(x_3)\frac{x_1^2-x_2^2}{2}),\\
		d\na\frac{\rh(x_3)(x_1^2-x_2^2)}{2}&=(d\rh)L+\rh dL+L\du \na\rh(x_3)+\frac{x_1^2-x_2^2}{2}d\na\rh(x_3),
	\end{align*} where $L\du=d\frac{x_1^2-x_2^2}{2}$ denotes the dual of the gradient. It is clear that $R$ only intersect the $y_1y_2x_3$-plane, i.e., $x_1=x_2=y_3=0$ along $\ga.$ On $\ga,$ we have $x_1=x_2=0,$ and thus $L=0,$ and $$d\na\frac{\rh(x_3)(x_1^2-x_2^2)}{2}=\rh dL.$$ This implies that the tangent plane to any point $q$ along $\ga$ on $R$ is parametrized by
	\begin{align}\label{tsp}
		T_qR=\gr \rh(x_3)\na \frac{x_1^2-x_2^2}{2}.
	\end{align} 
	This implies $T_qR$ is special Lagrangian at $q\in \ga.$ (Just note that by splitting off $x_3$-axis and changing into $w_1=x_1+ix_2,w_2=y_1-iy_2$ coordinate we get $w_2=\rh w_1$.)
	Moreover, for any point $q'$ on $R$ near $\ga$, its tangent plane's distance to the special Lagrangian plane $(d\rh(q'_3)+\rh(q'_3))\gr \rh(x_3)\na \frac{x_1^2-x_2^2}{2}$ can be controlled by $r\no{\rh}_{C^1}+r^2\no{\rh}_{C^2},$ where $r$ is the distance to $\ga.$ Thus, if we make $r$ small, then $T_{q'}R$ can be made $C^0$ close to being special Lagrangian. 
	\subsection*{Step 3}
	Again, $\ga$, the tangent planes to $\R$ are $O(r)$ close to being special Lagrangian. Thus, by shrinking the $r$ in $B_r(\ga),$ and using an implicit function argument, we can assign an $h_q\in U(3)$ to any point $q\in R,$ which depends smoothly on $R.$ $h_q$ is simply the $U(3)$ action that sends $x_1,x_2,x_3$-plane to the tangent plane of $R$ at $q.$ Moreover, $h_q\in SU(3)$ naturally for $q\in \ga,$ since by (\ref{tsp}), for $q\in\ga,$ $T_q R$ is a special Lagrangian plane. Now, apply Lemma \ref{ncs} to $R$ with $n_1=(0,0,0,1,0,0),n_2=(0,0,0,0,1,0),n_3=(0,0,0,0,0,1)$. Extend $h_q$ constantly along all $Y_j$ directions. Now consider the form $\ov{\phi}_q=(h_q\m)\du\phi.$ Since $h_q\in SO(6),$ $\ov{\phi}$ is still a calibration form. Moreover, we claim that for $\det h_q\m=e^{i\ta},$ we have \begin{align}\label{hzero}
		d\ta|_P\equiv 0.
	\end{align} Again, this is by tedious calculations similar to the one in Lemma \ref{ted} and we leave it to Appendix \ref{tedc}. By (\ref{dta})This implies $d\ov{\phi}=-d\ta\w (\sin \ta \phi+\cos\ta J\du\phi)$ is also $0$ on  $P.$
	\subsection*{Step 4}
	Apply Lemma \ref{inte} to the $Q$ coordinate system of $R$ to get a closed form $\psi$ so that it is equal to the special Lagrangian form $\phi$ on $P,$ on $\Si$ and outside of $B_r(\ga)\cap \{\frac{3}{2}r_0<x_3<2p_3-\frac{3}{2}r_0\}.$ Moreover, $\no{\psi-\ov{\phi}}\le \no{d\ov{\phi}}_{C^0}$. By construction $\no{d\ov{\phi}}_{C^0}\le \no{d\ta}_{C^0}.$ By III.2.D (2.19) in \cite{HLg}, $\no{d\ta}_{C^0}\le \no{H}_{C^0}.$ By smoothness of $R,$ as long as we shrink $r$ small enough, we can guarantee that $\no{\psi-\no{\phi}}_{C^0}\le \e$ for any $\e$ we want. Since $\no{\ov{\phi}-\phi}_{C^0}=O(r)$, we can make $\no{\psi-\phi}_{C^0}$ as small as we want. Thus, we can apply Lemma \ref{ift} to deduce the existence $h_q'\in GL(6)$, which depends smoothly on $q,$ so that $\psi=(h_q')\du \phi.$ Now make the new metric $g_q=((h_q')\m)\du\de,$ and we are done.
	
	\appendix     
	\section{Basic facts about immersion of Riemannian submanifolds}\label{bas}
	First, we recall the following facts for calculating the second fundamental form. We follow Einstein's summation conventions.
	
	Let $f$ an immersion from the unit ball $B_1^m(0)\s \R^m$ to $\R^{m+k}$ be an immersion. We will calculate the second fundamental of $f(B_1(0))$ pointwise in terms of the coordinate in $\R^m.$ First, let $v_j=f\pf(\pd_j)$ be the pushforward of the basis on $\R^m.$ We have $v_j=(\pd_j f^\ai)\pd_\ai'.$ (Here we use $\pd_j$ to denote the standard basis on $R^m$ and $\pd_\ai'$ be the standard basis on $\R^{m+k}$.) The metric coefficients in $v_j$ are $g_{ij}=\ri{v_i,v_j}.$ Now, it suffices to calculate the covariant derivative $\na_{v_i}v_j.$ We will use $[f(p)]$ to denote a function (or vector) that is evaluated at (the tangent space of) that point. For every point $x\in B_1^m(0),$ we have
	\begin{align*}
		\na_{v_i[f(x)]}v_j[f(x)]=&\na_{\pd_i f^\ai[x] \pd_{\ai}'[f(x)]}\pd_j f^\be[x] \pd_{\be}'[f(x)]\\
		=&\pd_i f^\ai[x] \pd'_\ai[f(x)](\pd_j f^\be[x])\pd'_\be[f(x)]+\pd_i f^\ai[x] \pd_j f^\be[x]\na_{\pd'_\ai[f(x)]}\pd'_\be[f(x)]\\
		=&f\pf\pd_i[f(x)](\pd_j f^\be[x])\pd'_{\be}[f(x)]+0\\
		=&\pd_i[x](\pd_j f^\be[x])\pd'_{\be}[f(x)]\\
		=&\pd_i\pd_j f^\be\pd'_\be.
	\end{align*}
	Thus, for any normal vector $n,$ the second fundamental form $A(v_i,v_j)$ at $f(x)$ of $f(B_1(0))$ is
	\begin{align*}
		\ri{A(v_i,v_j),n}=&\ri{\na_{v_i[f(x)]}v_j[f(x)],n}\\=&\ri{\pd_i\pd_j f^\be\pd'_\be,n^\ai \pd'_\ai}\\
		=&\pd_i\pd_jf^\be n^\be.
	\end{align*}
	The mean curvature $H$ is
	\begin{align*}
		\ri{H,n}=&g^{ij}\ri{A(v_i,v_j),n}=g^{ij}\pd_i\pd_j f^\be n^\be.
	\end{align*}
	\subsection{Equation (\ref{hzero})}\label{tedc}Again by III.2.D (2.19)  in \cite{HLg}, it suffice to verify that the mean curvature of $R$ at any point on $\ga\cap\{r_0<x_0<2p_3-r_0\}$ is zero, i.e.,
\begin{align}\label{mcz}
	H_R|_{\ga\cap\{r_0<x_0<2p_3-r_0\}}=0.
\end{align}
Here the immersion of $R$ is defined as
\begin{align*}
	f(x_1,x_2,x_3)=(x_1,x_2,x_3,\rh(x_3) x_1,-\rh(x_3) x_2,\rh'(x_3)\frac{x_1^2-x_2^2}{2}).
\end{align*}
We have the basis vectors,
\begin{align*}
	v_1=&(1,0,0,\rh(x_3),0,\rh'(x_3)x_1),\\
	v_2=&(0,1,0,0,-\rh(x_3),-\rh'(x_3)x_2),\\
	v_3=&(0,0,1,\rh'(x_3)x_1,-\rh'(x_3)x_2,\rh''(x_3)\frac{x_1^2-x_2^2}{2}).
\end{align*}
Along $\ga,$ we have $x_1=x_2=0.$ For any point $q$ on $\ga,$ this implies that
$v_1,v_2,v_3$ is an orthogonal basis with $\no{v_1}=\no{v_2}=\sqrt{1+\rh^2},\no{v_3}=1$. Moreover, the normal space to $T_qR$ is generated by
\begin{align*}
	v_4=&(-\rh,0,0,1,0,0),\\
	v_5=&(0,\rh,0,0,1,0),\\
	v_6=&(0,0,0,0,0,1).
\end{align*}
We have $\pd_j^2f^k=0$ for $1\le j,k\le 3,$ and $1\le j\le 2,4\le j\le 5$ and
\begin{align*}
	\pd_3^2f^4=\rh' x_1,\pd_3^2f^5=-\rh' x_2,\pd_1^2f^6=\rh',\pd_2^2f^6=-\rh',\pd_3^2f^6=\rh^{(3)}\frac{x_1^2-x_2^2}{2}.
\end{align*}
When restricting to $\ga,$ then only nonzero term is $\pd^2_1f^6,\pd_2^2 f^6.$
The only nonzero terms of $g^{ij}$ on $\ga$ is $g^{11}=g^{22}=\frac{1}{\sqrt{1+\rh^2}},g^{33}=1$. Thus, we have $	\ri{H,v_4}=\ri{H,v_5}=0,$ as $v_4,v_5$ has no sixth component. 
\begin{align*}
	\ri{H,v_6}=g^{11}\pd_1^2f^6+g^{22}\pd_2^2f^6=0.
\end{align*}
This implies $H=0$ along $\ga$.
	\section{Proof of Lemma \ref{ift}}\label{cif}
	Recall that we only need to calculate $\dim d\ker\lam$, the dimension of the Lie algebra of the invariant group of $\phi$, with $\phi$ being the special Lagrangian $3$-form,
	\begin{align*}
		\phi=\Re dz_1\w dz_2\w dz_3=dx_1\w dx_2\w dx_3-dy_1\w dy_2\w dx_3-dx_1\w dy_2\w dy_3-dy_1\w dx_2\w dy_3.
	\end{align*}
	To simplify the notation, sometimes we will also use $\pd_j$ to denote $\pd_{x_j}$ and use Roman numerals to indicate $y_j$-vectors. For example $\pd_{II}\equiv\pd_{y_2}.$ 
	We will let the coordinate be $(x_1,\cd,x_3,y_1,\cd,y_3),$ with identification $z_j=x_j+iy_j.$ We will number the coordinates by $1,2,3,I,II,III.$ Suppose $d\lam(h)\equiv0.$ Note that $d\lam(h)$ is still an exterior form, so we only have to verify the zero condition on the basis of $\bigwedge^3(\R^6).$ We will list the $20$ equations below. The notation is as follows $h_{1,III}$ means the $\ri{h\pd_{y_3},\pd_{x_1}}.$ The four basis in $\phi$ are $$123,-1 (II) III,-I2III,-I(II)3.$$ The equations are
	\begin{align*}
		123:&h_{1,1}&+h_{2,2}&+h_{3,3}&&&&=0,\\
		12I:&h_{III,1}&&&+h_{3,I}&&&=0,\\
		12II:& &h_{III,2}& &&+h_{3,II}&&=0,\\
		12III:&-h_{I,1}&-h_{II,2}&&&&+h_{3,III}&=0,\\
		13I:&-h_{II,1}&&&+h_{2,I}&&&=0,\\
		13II:&h_{I,1}&&+h_{III,3}&&-h_{2,II}&&=0,\\
		13III:&&&-h_{II,3}&&&-h_{2,III}&=0,\\
		1I(II):&-h_{3,1}&&&+h_{III,I}&&&=0,\\
		1I(III):&h_{2,1}&&&-h_{II,I}&&&=0,\\
		1II(III):&-h_{1,1}&&&&-h_{II,II}&-h_{III,III}&=0,\\
		23I:&&-h_{II,2}&-h_{III,3}&+h_{1,I}&&&=0,\\
		23II:&&h_{I,2}&&&-h_{1,II}&&=0,\\
		23III:&&&h_{I,3}&&&-h_{1,III}&=0,\\
		2I(II):&&-h_{3,2}&&&+h_{III,II}&&=0,\\
		2I(III):&&h_{2,2}&&+h_{I,I}&&+h_{III,III}&=0,\\
		2(II)(III):&&-h_{1,2}&&&+h_{I,II}&&=0,\\
		3I(II):&&&-h_{3,3}&-h_{I,I}&-h_{II,II}&&=0,\\
		3I(III):&&&h_{2,3}&&&-h_{II,III}&=0,\\
		3II(III):&&&-h_{1,3}&&&-h_{1,III}&=0,\\
		I(II)III:&&&&-h_{1,I}&-h_{2,II}&-h_{3,III}&=0.
	\end{align*}
	One can verify that the only terms that appear twice are $$h_{1,1},h_{I,1},h_{2,2},h_{II,2},h_{3,3},h_{III,3},h_{1,I},h_{I,I},h_{2,II},h_{II,II},h_{3,III},h_{III,III}.$$ However, any term that appears twice involves two equations that only have that term in common. Thus, the rank is $20,$ equal to the number of equations. This verifies that $\dim\ker d\lam=\dim GL(\R^6)-\dim\bigwedge^3(\R^6).$
	\begin{rem}
		Professor Camillo De Lellis has pointed out that the same argument as above works for the associative $3$-form. What happens for the other cases? For $k=0,n,$ we have not much to say, as the corresponding degree exterior forms are of only dimension $1.$ For $k=2,$ this dimension counting is true for symplectic forms and partially explains the Darboux theorem. Let $\kappa(k,n)=\dim GL(\R^n)-\dim\bigwedge^k(\R^n).$ $\kappa>0$ is indeed a somewhat rare condition for large $k$, as $\dim \bigwedge^k(\R^n)\sim n^{n-k}/k!,$ while $\dim GL(\R^n)=n^2.$ Indeed we have
		\begin{align*}
			\kappa>0 \textnormal{ if and only if }\begin{cases}
				n\le 7,\\
				\text{or }k=0,1,2,n-2,n-1,n,\\
				\text{or }n=8,k\not=4.
			\end{cases}
		\end{align*} 
	\end{rem}

	\section{Calculations of Lemma \ref{rays}}
	The calculations are done as follows. Since $C$ is conical, to calculate the number of intersecting rays of $C$ and any special Lagrangian plane, it suffices to consider the intersection points of the link $F$ with the plane. Moreover, by Theorem 0.2 in \cite{DS1}, $F$ can only intersect a special Lagrangian plane along finitely many points.
	
	Note that $F$ is the variety determined by the following equations
	\begin{align}\label{dfe}
		\begin{cases}
			z_1z_2z_3=\frac{1}{3\sqrt{3}},\\
			|z_1|^2=|z_2|^2=|z_3|^2=\frac{1}{3}.
		\end{cases}
	\end{align}
	Thus, we have to calculate the number of points on the plane $P(\tau,\ta).\pi_0$ satisfying (\ref{dfe}).
	
	The plane $P(\tau,\ta).\pi_0$ is spanned by its rows, so any vector therein can be represented as
	\begin{align*}
		v=&\begin{pmatrix}
			a&b&c
		\end{pmatrix}.P(\tau,\ta).\pi_0\\=&\bigg(\frac{a}{\sqrt{3}}+\frac{i \left(b e^{i t} \cos (s)-c e^{-i t} \sin
			(s)\right)}{\sqrt{2}}+\frac{i \left(b e^{i t} \sin (s)+c e^{-i t} \cos
			(s)\right)}{\sqrt{6}}\\&\frac{a}{\sqrt{3}}-i \sqrt{\frac{2}{3}} \left(b e^{i t} \sin (s)+c e^{-i
			t} \cos (s)\right)\\&\frac{a}{\sqrt{3}}-\frac{i \left(b e^{i t} \cos (s)-c e^{-i t} \sin
			(s)\right)}{\sqrt{2}}+\frac{i \left(b e^{i t} \sin (s)+c e^{-i t} \cos
			(s)\right)}{\sqrt{6}}\bigg),
	\end{align*}where $a,b,c\in\R.$
	We only have to substitute the coordinates of $v$ into \ref{dfe} and look for solutions. Moreover, note that $\no{v}^2=a^2+b^2+c^2$, and $F$ sits inside the unit ball. Thus, it suffices to look for $|a|,|b|,|c|<2.$ The rest is a very tedious calculation. We will provide a Mathematica program that can verify our calculations. The code is attached at the end of the document.

	\section{Mathematica verifications (see attachments)}\label{mat}

\includepdf[pages=-]{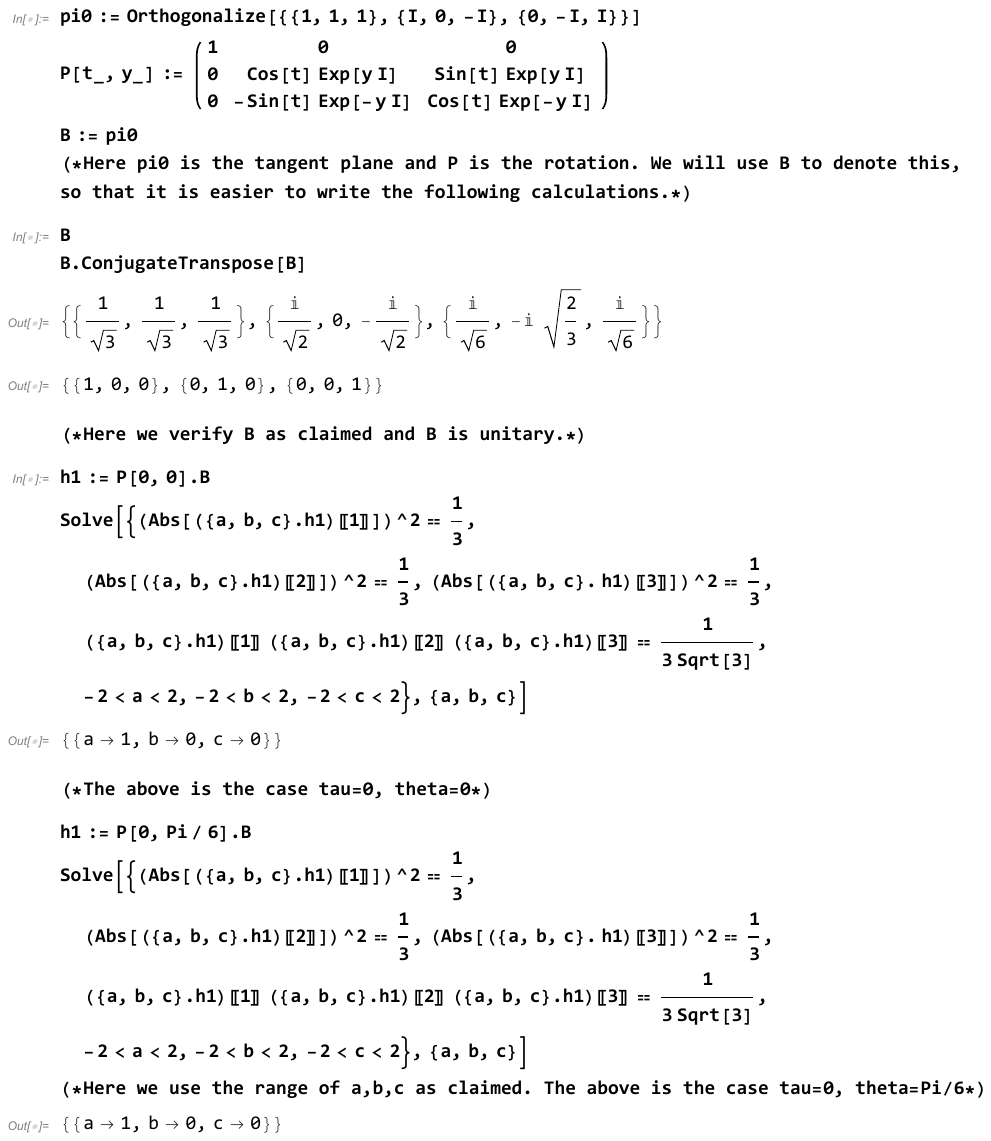}


\begin{thebibliography}{9}
		\bibitem{A}Frederick J. Almgren, Jr. \textit{Almgren's big regularity paper. Q-valued functions minimizing Dirichlet's integral and the regularity of area-minimizing rectifiable currents up to codimension 2.} With a preface by Jean E. Taylor and Vladimir Scheffer. World Scientific Monograph Series in Mathematics, 1. World Scientific Publishing Co., Inc.
		\bibitem{At}Frederick J. Almgren, Jr. \textit{The homotopy groups of the integral cycle groups.} Topology 1 (1962), 257–299.
		\bibitem{RB}Robert L. Bryant, private communication.
		\bibitem{RB1}Robert L. Bryant,
		\textit{$SO(n)$-Invariant Special Lagrangian Submanifolds of $\C^{n+1}$ with Fixed Loci}, Chinese Annals of Mathematics, Series B volume 27, pages95–112 (2006)
		\bibitem{SC}Sheldon Xu-Dong Chang, \textit{Two-dimensional area minimizing integral currents are classical minimal surfaces}, J. Amer. Math. Soc. 1 (1988), no. 4, 699–778.
		\bibitem{DL}Camillo De Lellis, private communication.
		\bibitem{DPH}C. De Lellis; G. De Philippis; J. Hirsch, \textit{Nonclassical minimizing surfaces with smooth boundary}, To appear in Journal of Differential Geometry
		\bibitem{DPHM}C. De Lellis; G. De Philippis; J. Hirsch; A. Massaccesi, \textit{On the boundary behavior of mass-minimizing integral currents}, available at https://www.math.ias.edu/delellis/node/148,
		\bibitem{DHMS}C. De Lellis; J. Hirsch; A. Marchese; S. Stuvard, \textit{Regularity of area minimizing currents mod p}, to appear in Geometric and Functional Analysis
		\bibitem{DS1}C. De Lellis; E. Spadaro, \textit{Regularity of area-minimizing currents I: $L^p$ gradient estimates}, Geom. Funct. Anal. 24 (2014), no. 6, 
		\bibitem{DS2}C. De Lellis; E. Spadaro, \textit{	Regularity of area-minimizing currents II: center manifold},
		Ann. of Math. (2) 183 (2016), no. 2, 499–575.
		\bibitem{DS3}C. De Lellis; E. Spadaro, \textit{Regularity of area-minimizing currents III: blow-up},
		Ann. of Math. (2) 183 (2016), no. 2, 577–617.
		Geom. Funct. Anal. 24 (2014), no. 6, 1831–1884.
		\bibitem{DSS0} C. De Lellis; E. Spadaro; L. Spolaor, \textit{Uniqueness of tangent cones for 2-dimensional almost minimizing currents}, Comm. Pure Appl. Math. 70, 1402-1421
		\bibitem{DSS1}		C. De Lellis; E. Spadaro; L. Spolaor, \textit{Regularity theory for 2-dimensional almost minimal currents I: Lipschitz approximation},	Trans. Amer. Math. Soc. 370 (2018), no. 3, 1783–1801
		\bibitem{DSS2}		C. De Lellis; E. Spadaro; L. Spolaor, \textit{Regularity theory for 2-dimensional almost minimal currents II: branched center manifold},	Ann. PDE 3 (2017), no. 2, Art. 18, 85 pp.
		\bibitem{DSS3}C. De Lellis; E. Spadaro; L. Spolaor, \textit{Regularity theory for 2-dimensional almost minimal currents III: blowup}		To appear in Jour. Diff. Geom.
		\bibitem{HF} Herbert Federer, \textit{Geometric Measure Theory} Springer, New York, 1969.
		\bibitem{HF1}Herbert Federer, \textit{The singular sets of area minimizing rectifiable currents with codimension one and of area minimizing flat chains modulo two with arbitrary codimension.} Bull. Amer. Math. Soc. 76 (1970), 767–771. 		
		\bibitem{AG} Alfred Gray \textit{Tubes.} Second edition. With a preface by Vicente Miquel. Progress in Mathematics, 221. Birkhäuser Verlag, Basel, 2004.
		\bibitem{BH}Brian C. Hall, \textit{Lie Groups, Lie Algebras, and Representations: An Elementary Introduction}, Graduate Texts in Mathematics, 222 (2nd ed.), Springer, 2015
		\bibitem{HLg}Reese Harvey; H. Blaine Lawson, Jr. \textit{Calibrated geometries}. Acta Math. 148 (1982),
		\bibitem{HLf}Reese Harvey; H. Blaine Lawson, Jr. \textit{Calibrated foliations (foliations and mass-minimizing currents)}. Amer. J. Math. 104 (1982), no. 3, 607–633.
		\bibitem{AH}Allen Hatcher, \textit{Algebraic topology}, Cambridge University Press, Cambridge, 2002.
		\bibitem{DJ}Dominic Joyce; \textit{Lectures on special Lagrangian geometry.} Global theory of minimal surfaces, 667–695, Clay Math. Proc., 2, Amer. Math. Soc., Providence, RI, 2005.
		\bibitem{ZL}Zhenhua Liu, \textit{On a conjecture of Almgren: area-minimizing surfaces with fractal singularities}, preprint available at \url{arxiv.org/abs/2110.13137}
		\bibitem{ZL1}Zhenhua Liu, \textit{Homologically area-minimizing surfaces with non-smoothable singularities}, preprint available at \url{arxiv.org/abs/2206.08315} 
		\bibitem{JL}John M. Lee, Introduction to smooth manifolds. Second edition. Graduate Texts in Mathematics, 218. Springer, New York, 2013.
		\bibitem{FM}Frank Morgan, \textit{On the singular structure of two-dimensional area minimizing surfaces in $R^n$.} Math. Ann. 261 (1982), no. 1
		\bibitem{NV}Aaron Naber, Daniele Valtorta, \textit{The singular structure and regularity of stationary varifolds}, J. Eur. Math. Soc., Volume 22, Issue 10, 2020
		\bibitem{BO} Barrett O'Neill, \textit{Semi-Riemannian geometry}. With applications to relativity. Pure and Applied Mathematics, 103. Academic Press, Inc.	
		\bibitem{LS1} Leon Simon, \textit{Lectures on Geometric Measure Theory,} Proceedings for the Centre for Mathematical Analysis, Australian National University, Canberra, 1983.
		\bibitem{LS}Leon Simon, \textit{Stable minimal hypersurfaces in $\R^{N+1+l}$ with singular set an arbitrary closed K in $0\times\R^l$}, available at \url{https://arxiv.org/abs/2101.06401}
		\bibitem{RT}R. Thom, \textit{Quelques propriétés globales des variétés différentiables}, Comment. Math. Helv. 28 (1954), 17–86.
		\bibitem{CWd} C. T. C. Wall, \textit{Differential topology}, Cambridge Studies in Advanced Mathematics, 156. Cambridge University Press, Cambridge, 2016. 
		\bibitem{BW}Brian White, \textit{Tangent cones to two-dimensional area-minimizing integral currents are unique}, Duke Math. J. 50 (1983), no. 1, 143–160.
		\bibitem{HW}Hassler Whitney, \textit{Differentiable manifolds.} Ann. of Math. (2) 37 (1936), no. 3, 645–680.
		\bibitem{YZa}Yongsheng Zhang, \textit{On extending calibration pairs.} Adv. Math. 308 (2017), 645–670.
		\bibitem{YZj}Yongsheng Zhang \textit{On realization of tangent cones of homologically area-minimizing compact singular submanifolds.} J. Differential Geom. 109 (2018), no. 1, 177–188. 	
		\bibitem{PL}Poincaré lemma, nLab, available at \url{https://ncatlab.org/nlab/show/Poincar%C3%A9+lemma}
	\end{thebibliography}
\end{document}